\documentclass[11pt,english]{article}
\usepackage[cp1251]{inputenc}
\usepackage{amssymb,amsmath,babel,amsthm,graphicx,cite}

\usepackage{tikz-cd,multicol}

\newcounter{parag}

\newtheorem*{theorem*}{Theorem}
\newtheorem{lemma}{Lemma}
\newtheorem{prop}{Statement}

\theoremstyle{definition}

\newcounter{cl}

\title{Recognizing $A_7$ by its set of element orders} 

\author{E. Jabara -- A. Mamontov } \date{}

\begin{document}

\maketitle

\noindent{\bf Abstract. } {\it Let $G$ be a periodic group, the spectrum $\omega(G) \subseteq \mathbb{N}$ of $G$ is the set of orders of elements in $G$. In this paper we prove that the alternating group $A_{7}$ is uniquely defined by its spectrum in the class of all groups.} \medskip
	
\noindent{\bf Keywords:} periodic group, locally finite group, spectrum.
	
\begin{section}{Introduction}
	
Let $\mathfrak{M}$ be a class of periodic groups and $G \in \mathfrak{M}$.  The {\it spectrum} of $G$ is the set $$\omega (G)=\{ n \in \mathbb{N} \mid n \mbox{ is the order of some element in } G \} $$ and $\mu(G)$ is the set of maximal elements of $\omega(G)$ with respect to division. In particular $\omega(A_7)=\{1,2,3,4,5,6,7\}$ and $\mu(A_7)=\{4,5,6,7\}$. A group $G$ is called {\it 	recognizable by spectrum in} $\mathfrak{M}$
if for any ${H \in \mathfrak{M}}$ the equality $\omega(H)=\omega(G)$ implies isomorphism $H \simeq G$.

Many finite simple groups are recognizable by spectrum in the class of finite groups (see \cite{fin_review} for a survey of known results). Moreover, $L_2(2^m)$ \cite{zhma1999e}, $L_2(7) \simeq L_3(2)$ \cite{lyku1}, Mathieu group $M_{10}$ \cite{jlm2014e}, and $L_3(4)\simeq M_{21}$ \cite{l34e} are known to be recognizable by spectrum in the class of periodic groups. There are examples of finite simple groups that are recognizable by spectrum in the class of finite groups but not recognizable in the class of periodic groups \cite{pernotfin}, which are related to non-locally finite groups of large even exponent that provide negative solution to the Burnside problem. In the paper we prove

\vspace{3mm}
	
{\bf Theorem.} {\it $A_7$ is recognizable by spectrum in the class of periodic groups.}

\vspace{3mm}

Note that \cite{OCn} contains the classification of finite $OC_n$ groups, i.e. groups with spectrum $\{1,2,\ldots, n\}$ for some $n \in \mathbb{N}$. In particular, it is shown that $A_7$ is recognizable by spectrum in the class of finite groups. Therefore our major scope is to prove that a group $G$ with $\omega(G)=\omega(A_7)$ is locally finite.

For $n \leq 6$ it is known that $OC_n$ groups are locally finite \cite{neu1937,sane,maz1-5,lmmj2014e}. Theorem gives a positive solution to \cite[Conjecture 7.2.1]{rev_per}. In \cite{OCn} it is proved that if $G$ is a finite group in $OC_{8}$, then $G\simeq L_{3}(4) \rtimes \langle \beta \rangle$, where $\beta$ is a unitary automorphism of $L_{3}(4)$, and that if $n\geq 8$ then there are no finite groups in $OC_n$. In this context it is interesting to highlight the following open problems.

\vspace{3mm}

{\bf Problem 1.} Is a group with spectrum $\{1,2,\ldots, 8\}$ locally finite?

\vspace{3mm}

{\bf Problem 2.} Is there a group with spectrum $\{1,2,\ldots, n\}$, where $n>8$?

\end{section}
	
\begin{section}{Notations, preliminary results and strategy}

Let $A_n$ and $S_n$ denote the alternating and symmetric groups of degree $n$ correspondingly. Let also $L_n(q)$ be the projective special linear group of dimension $n$ over the field with $q$ elements. The prime power $p^k$ denotes the elementary abelian group of order $p^k$, the number $n$ denotes the cyclic group of order $n$, $p^{1+2}$ denotes the extraspecial group of order $p^3$ with no elements of order $p^2$. The following groups are defined by generators and relations: 
$$F_{k^2 \cdot 6} = \big \langle x,t \; \big | \; x^3,t^2,(xt)^6,[x,t]^k \big \rangle \simeq k^2:6;$$
$$F_{294} = \big \langle x,t \; \big | \; x^3,t^2,(xt)^6,[x,t]^7 \big \rangle \simeq 7^2:6.$$
We have $F_{42} = \langle x,z \mid \rho_{42} \rangle$, where $$\rho_{42} =\{x^3,z^2,(xz)^6,b^7,b^x=b^4 \} \text{ and }b=z^xz.$$
$$F_{36}=\big \langle t,x \mid t^4,x^3,(t^2x)^2,[x,x^t] \big \rangle;$$
$$3^{1+2}:2=\langle x,t \mid x^3,t^2,(xt)^6,[x,t]^3 \rangle.$$

We remark that $F_{42} < F_{294}$ with $z=t^{xt}$.

We denote by $\Gamma_n=\Gamma_n(G)$ the set of elements of order $n$ in $G$ and \[ \Delta=\Delta(G)=\{x^2\mid x \in \Gamma_4 \}.\] If $p$ is a prime then $O_{p}(G)$ is the largest normal $p$-subgroup of $G$. If $A$ and $B$ are groups, $A:B$ denotes some extension of $A$ by $B$. Some local notations are introduced at the beginning of paragraphs.

Speaking of computations we refer to computations in {\sc Gap} \cite{gap} using coset enumeration algorithm. 

Further we assume that a group $G$ with  $\omega(G)=\omega(A_7)=\{1,2,3,4,5,6,7\}$ is a counterexample to Theorem such that the exponent of $O_2(G)$ is the smallest possible.

\begin{lemma}\label{l:nonlf}
$G$ is not locally finite.
\end{lemma}

\begin{proof} Assume that $G$ is locally finite. Choose elements $x_4,x_5,x_6$, and $x_7$ in $G$ so that the order of $x_i$ is $i$. Then $H = \langle x_4,x_5,x_6,x_7 \rangle$ is finite and $\omega(H)=\omega(G)$. By \cite{OCn} $H \simeq A_7$. If $G \not = H$ then take $x \in G \setminus H$. Again $\langle H,x \rangle$ is finite, and hence isomorphic to $A_7$. Therefore $H=\langle H,x \rangle = G$, a contradiction.
\end{proof}

A corollary of Lemma \ref{l:nonlf} and Shunkov's Theorem \cite{shu1972e} is

\begin{lemma}\label{l:shunk}
The centralizer of every involution in $G$ is infinite.
\end{lemma}

Our general strategy is to show first that $G$ contains some nonabelian finite simple subgroup $K$ with an involution $a \in K \cap \Delta$, and then approach this case via $C_G(a)$ as we describe at the beginning of paragraph 6.

The first step was started at \cite{m2020e}, where the following statement was proven:

\begin{prop}\label{s:arxiv}
Let $H$ be a subgroups of $G$ isomorphic to $A_4$, whose involutions are in 
$\Delta(G)$. Then either $O_2(H) \subseteq O_2(G)$, or $G$ has a finite simple subgroup $A_5$ or $L_2(7)$, whose involutions are in $\Delta(G)$. 
\end{prop}

{\bf Remark:} Let $H=V:A_5$, where $V=2^4$. Then $\mu(H) = \{4,5,6\}$. Involutions of $\Delta(H)$ generate $V=O_2(H)$, and involutions, which are not in $\Delta(H)$ lie in a subgroup isomorphic to $A_5$. This example shows the significance of the condition $O_2(H) \subseteq \Delta$.

\end{section}

\begin{section}{On factors with no elements of order 4}
	
Throughout the paragraph $T$ is a group and $\mu(T) =\{5,6,7\}$. Let $$\Gamma^{\ast}_2=\Gamma^{\ast}_2(T)=\{t\in \Gamma_2(T)  \mid C_T(t) \mbox{ is elementary abelian}\, \}.$$ We assume that this set is non-empty and we choose $a \in \Gamma^{\ast}_2$. Let $$\Lambda_2 = \Lambda_2(T) =\{x^3 \mid x \in \Gamma_6(T) \}. $$ Note that $\Gamma^{\ast}_2$ and $\Lambda_2$ are two normal non-intersecting sets of involutions and $\Lambda_2 \not = \emptyset$. The goal of this paragraph is to prove that 
$\langle \Gamma^{\ast}_2 \rangle$ is a group of exponent $6$.
	
\begin{lemma}\label{l:2inv}
If $b\in \Gamma_2(T)$, then $(ab)^6=1$. In particular, $[a,t]^6=1$ for any $t \in T$.
\end{lemma}

\begin{proof}
Assume the contrary. If the involutions $a$ and $b$ are not conjugated, then the order of $ab$ is even, and  conclusion of Lemma is true. Therefore we assume $b \in  \Gamma^{\ast}_2$.

Let $c \in \Lambda_2$. Then the order of $ac$ is even, and if it equals $6$, then the involution $(ac)^3$ from the center of dihedral subgroup $\langle a,c \rangle$ is in $\Lambda_2$. Therefore $a$ commutes with some element of $\Lambda_2$, and we may assume that $[a,c]=1$. So the following relations hold in $T$ ($p$ odd): $$\tau(p)=\{a^2,b^2,c^2,(ab)^p,[a,c]\}.$$ Sets $\Gamma^{\ast}_2$ and $\Lambda_2$ are normal and do not intersect, so the following relations hold: $$\rho_1=\{(bc)^6,(a^bc)^6,(b^{ab}c)^6\}.$$ Applying similar argument once again, we obtain the relations $$\rho_2=\{(a(bc)^3)^6,(a(a^bc)^3)^6, (a(b^{ab}c)^3)^6\}.$$ It follows that $\langle a,b,c \rangle$ is a homomorphic image of $$T(p,i,j)=\big \langle a,b,c \; \big | \;  \rho_1 \cup \rho_2 \cup \tau(p) \cup \sigma(i,j) \,\big \rangle,$$ where $\sigma(i,j)=\{(abc)^i,(aa^bc)^j\}$. Computations show that the order of $T(p,i,j)$ divides $4$ for all $p \in \{5,7\}$ and $i,j \in \{5,6,7\}$: a contradiction.
\end{proof}
	
\begin{lemma}\label{l:23T}
If $x \in \Gamma_3(T)$, then $(ax)^6=1$.	
\end{lemma}

\begin{proof} Lemma \ref{l:2inv} implies $[a,x]^6=1$. Therefore $\langle a,x \rangle $ is a homomorphic image of $K(\ell)=\langle a,x \mid a^2,x^3,(ax)^{\ell},[a,x]^6 \rangle$, $\ell \in \{5,6,7\}$. Computations show that $K(5)=1$ and $K(7)\simeq L_2(13)$, which is not possible. So $\ell=6$ as required.	
\end{proof}

\begin{lemma}\label{l:comm}
If $b,c \in \Gamma^{\ast}_2(T)$, then $[(ab)^2,(bc)^2]=1$.
\end{lemma}

\begin{proof} First assume that orders of $ab$ and $bc$ divide $3$ and hence the following relations hold: $$\rho=\{a^2,b^2,c^2,(ab)^3,(bc)^3\}.$$ 
Using Lemmas \ref{l:2inv} and \ref{l:23T} we obtain the following set of relations
$$\tau=\{(ac)^6,(a^bc)^6,(ab \cdot c)^6, (b\cdot (ac)^3)^6\}.$$ Therefore $\langle a,b,c \rangle $ is a homomorphic image of $$K=\big \langle a,b,c \; \big | \; \rho\cup \tau \big \rangle. $$ Computations show that $K\simeq 3^{1+2}:2$. 
The center of $K$ has order $3$ and so it is contained in the kernel of the corresponding homomorphism. It follows that $\langle a,b,c \rangle$ is an extension of an elementary abelian $3$-group by an involution. In particular, $[ab,bc]=1$.

Let now $b,c$ be arbitrary elements of $\Gamma^{\ast}_2(T)$. Then $b,b^a,b^c \in  \Gamma^{\ast}_2(T)$ and $(b^ab)^3=(bb^c)^3=1$. We have shown that $[(ab)^2,(bc)^2]=[b^ab,bb^c]=1$. 
\end{proof}

\begin{lemma}\label{l:4inv}
If $b,c \in \Gamma^{\ast}_2(T)$, then $(abc)^2=1$. Moreover, $[a'b,cd]=1$ for any $a',b,c,d \in \Gamma^{\ast}_2(T)$.
\end{lemma}

\begin{proof}
Let $w=(abc)^2$ and let
$$ \sigma=\{ a^2,b^2,c^2,(ab)^6,(bc)^6,(ac)^6,[(ab)^2,(bc)^2],[(bc)^2,(ca)^2],[(ca)^2,(ab)^2 \},$$
by Lemma \ref{l:comm} $\langle a,b,c \rangle$ is a homomorphic image of $$T(i)= \big \langle a,b,c \mid \sigma \cup \{(abc)^{i} \} \big \rangle, \;\;\;\; i \in \{5,6,7\}.$$ Computations show that $|T(5)|=|T(7)|=4$, therefore $i=6$ and the order of $w\in T$ divides $3$. With further computations we can prove that in $T(6)$ the element $ww^a$ has order $3$ and centralizes $a$. So in $T$ we have $w^a=w^{-1}$. Similarly $w^b=w^{-1}$ and $w^c=w^{-1}$. It follows that $w=w^{abc}=w^{-1}$. The order of $w$ divides $3$, so $w=1$.

In order to prove that $[a'b,cd]=1$, we can assume $a'=a$. We have $(abc)^2=1=(abd)^2$. Consequently, $(ab)^c=(ab)^{-1}=(ab)^d$ and therefore $(ab)^{cd}=ab$. 
\end{proof}

\begin{lemma}\label{l:no4}
Let $T$ be a group and $\mu(T)=\{5,6,7\}$. Then $\langle \Gamma^{\ast}_2(T) \rangle$ is a locally finite group of period $6$.
\end{lemma}

\begin{proof}
Let $a_1, a_2, \ldots, a_n \in \Gamma^{\ast}_2(T)$ and let $H=\langle a_1, a_2, \ldots, a_n \rangle$. By Lemma \ref{l:4inv} the subgroup $K=\langle a_ia_j \mid i,j \in \{1,2,\ldots,n\} \rangle$ is normal and abelian of exponent 6. Therefore $H=K\langle a_1 \rangle$ is finite and satisfies lemma's conclusion. Any element of  $\langle \Gamma^{\ast}_2(T) \rangle$ can be written as a finite word of generators and this proves the lemma.
\end{proof}

\end{section}

\begin{section}{Reduction of the $2$-radical}
	
The goal of the paragraph is to prove the following
	
\begin{prop}\label{s:o2}
$O_2(G)=1$.
\end{prop}

\begin{proof}
Assume $N=O_2(G) \neq 1$ and let $\overline{G} = G /N$. 

By a well known result of Sanov \cite{sane}, $N$ is locally finite. Let $x \in \Gamma_5(G)$ 
then, by Schmidt's theorem, $\langle N,x \rangle$ is locally finite and by Thompson's theorem $N$
is locally nilpotent. By Higman's theorem \cite{higfree}, $N$ is nilpotent and, in particular, $Z(N) \not =1$.

We consider several cases.

{\bf 0.} $\omega(G)=\omega(\overline{G})$. By the choice of $G$, Statement's conclusion holds for $\overline{G}$, and so $\overline{G} \simeq A_{7}$.
By Schmidt's theorem $G$ is locally finite, a contradiction by Lemma \ref{l:nonlf}.

{\bf 1.}  $\mu (\overline{G}) = \{ 3,4,5,7\}$. Then, from \cite{l34e}, we have $\overline{G} \simeq L_3(4)$. By Schmidt's theorem $G$ is locally finite, a contradiction by Lemma \ref{l:nonlf}.

{\bf 2.} $\mu (\overline{G}) = \{ 5,6,7\}$. In this case $\Delta \subseteq N$. 

{\bf 2.1.}  Assume that there is $h \in G$ of order $4$ which is not in $N$. Then $\overline{h}$ is an involution in $\overline{G}$. 

We first prove that $C_{\overline{G}}(\overline{h})$ is an elementary abelian 2-group. Let $C$ be a full preimage of  $C_{\overline{G}}(\overline{h})$ in $G$. Assume that $C$ contains an element $x$ of order $3$, then $[x,h] \in N$. Hence $[x,N \langle h \rangle] \subseteq N$. Note that $\langle x,h,N \rangle$ is locally finite as it is an extension of a locally finite group $N$ by a  $(2,3)$-generated (and hence finite \cite[Lemma 2.1]{m2020e}) group $\langle \overline{x} , \overline{h} \rangle$. Changing $N$ to a minimal $\langle x,h \rangle$-invariant subgroup in $N$ that contains $N \cap \langle x,h \rangle$, we may assume that $\langle x,h,N \rangle$ is finite. Let $P=\langle h,N \rangle$. Since $[x,P] \subseteq N$ then $P$ is a finite $\langle x \rangle$-invariant $2$-subgroup of $\langle x,h,N \rangle$. Denote by $\Phi=\Phi(P)$ the Frattini subgroup of $P$. Note that $h^2 \in \Phi = P'\Delta$ and $P / \Phi$ is elementary abelian. Since $x$ is a nontrivial automorphism of $P$, it acts nontrivially on $P/\Phi$. Consequently all basis elements $h\Phi,h^x\Phi,h^{x^2}\Phi$ of $P / \Phi$ are in $[x\Phi,P/\Phi]$, and $C_{P/\Phi}(x)=1$. The equality $P/\Phi=[x\Phi,P/\Phi] \times C_{P/\Phi}(x\Phi)$ now imply that $P/\Phi = [x\Phi,P/\Phi] \subseteq N/\Phi$, hence $P \subseteq N$, a contradiction.

Therefore, $\overline{h} \in \Gamma^{\ast}_2(G/N)$ and $\overline{H}=\langle  \Gamma^{\ast}_2(G/N) \rangle$ is a locally finite group of exponent $6$ by Lemma \ref{l:no4}. Recall that $N=O_2(G)$, therefore $\overline{H}$ contains an element of order $3$. Let $H$ be the full preimage of $\overline{H}$ in $G$. Then $H$ is a normal locally finite $\{2,3\}$-subgroup of $G$, containing an element of order $3$ and an element $h$ of order $4$. Let $x \in \Gamma_3(H)$ and $y\in \Gamma_5(G)$. By Schmidt's theorem $H \rtimes \langle y \rangle$ is locally finite. Therefore $K=\langle x,h \rangle^K \rtimes \langle y \rangle$ is a finite group. An element $y$ acts on the subgroup $S=\langle x,h \rangle ^K$ fixed point freely. Therefore $S$ is nilpotent and so is a direct product of its Sylow subgroups, in particular, it contains an element of order $12$, which is not possible.

{\bf 2.2.} Let $\Gamma_4(G) \subseteq N$. 

If $a \in \Gamma_2(G)\backslash N$ and $b \in \Gamma_2(N)$, then $N\langle a \rangle$ is a $2$-group, and so $(ab)^4=1$. If $ab \in \Gamma_4 \subset N$, then $a \in N$, a contradiction. Therefore $(ab)^2=1$. Let $C=C_G(\Gamma_2(Z(N)))$. By construction $\Gamma_4 \subseteq N \leq C$. As already proven $\Gamma_2 \subseteq C$. The set $\Gamma_2(Z(N))$ is normal in $G$, therefore, $C \triangleleft G$.

We have $Z(N) \not =1$, therefore $C$ is contained in the centralizer of an involution, and so has no elements of orders $5$ and $7$. An element $x \in \Gamma_5(G)$ acts freely on $C$. Using arguments above deduce that $C$ is nilpotent. It follows that $C$ has no elements of order $3$, as $12 \not \in \omega(G)$. Therefore $O_2(G)=C \not = N$, a contradiction.

{\bf 3.} $\mu (\overline{G}) = \{ 2,3,5,7\}$. By a result of Mazurov \cite{maze} one of the following cases holds:

{\bf 3.1.} $\overline{G}$ is an extension of an abelian $2$-group $\overline{V}$ by a group with no involutions. Let $V$ be a full preimage of $\overline{V}$ in $G$. Then $O_2(G)=V \not =N$, a contradiction.

{\bf 3.2.} $\overline{G}$ is locally finite. By Schmidt's theorem $G$ is locally finite, a contradiction with Lemma \ref{l:nonlf}.

{\bf 4.} $\mu (\overline{G}) = \{ 3,5,7\}$.  Consider the following characteristic subgroups $L=\langle \Gamma_4 (N) \rangle$ and $D=\langle \Delta(L) \rangle$. Since $N$ is locally finite and nilpotent we have $L \not = D$. Let $\overline{G}=G /D$ and $x \in \Gamma_3(G)$. Using similar arguments as in 2.1 obtain $C_{\overline{L}}(\overline{x})=1$, i.e. $\overline{x}$ acts fixed point freely on $\overline{L} \not =1$. It follows that $\langle \overline{x} \rangle$ is a normal subgroup of $\overline{G}\,$  \cite[Theorem 3]{zh2000e}, and so it is necessarily central. A contradiction, since $15,21 \not \in \omega(G)$. 
\end{proof}	

\end{section}

\begin{section}{Existence of nonabelian finite simple subgroup}

The goal of the paragraph is to prove the following

\begin{prop}\label{fsg}
$G$ has a subgroup $H$ isomorphic to $A_5$ or to $L_2(7)$ such that $\Gamma_2(H) \subseteq \Delta$.
\end{prop}

Throughout the paragraph we assume the contrary, i.e. that {\it $G$ has no finite nonabelian simple subgroups $H$ with $\Gamma_2(H) \subseteq \Delta$.}	

\begin{lemma}\label{2+3-A4}
If $t \in \Delta$ and $x \in \Gamma_3$, then $(xt)^6=[x,t]^p=1$ and $p \in \{3,5,7\}$.
\end{lemma}

\begin{proof}
Let $H=\langle t,x \rangle$. All possible groups $H^{\ast}$
generated by an involution and element of order 3 and such that $\omega(H^{\ast}) \subseteq \omega(A_{7})$ are listed in \cite[Lemma 2.1]{m2020e}. By Proposition \ref{s:arxiv}, $H$ has no subgroups isomorphic to $A_4$, so $(xt)^6=1$ and by \cite[Lemma 9]{mamontov2013e} the order of $[x,t]$ is odd.
\end{proof}
	 
In the next lemmas we prove $p=3$.

\begin{lemma}\label{l:noF42}
$G$ cannot have a subgroup $H$ isomorphic to $F_{42}$ or to $F_{294}$ with $\Gamma_2(H) \subseteq \Delta(G)$.
\end{lemma}
	
\begin{proof}
Since $F_{42}<F_{264}$, it is sufficient to consider the case where $H$ is isomorphic to $F_{42}$. Assume the contrary. 

Let $F_{42} \simeq H= \langle
x,z \mid \rho_{42} \rangle \leq G$, where $z\in \Delta(G)$, $b=z^xz$ and $$\rho_{42} =\big \{x^3,z^2,(xz)^6,b^7,b^xb^{-4} \big \}.$$ Then $u=z^{x^2z} \in \Delta(G)$ centralizes $x$ and $C_{\langle x,z \rangle}(u)=\langle u,x \rangle$. We break the proof into several steps.

\vspace{1mm}

{\bf (1)} $C_G(x)= \langle x,u \rangle.$

Assume that $y$ is an element of order $3$ centralizing $x$, and $y\not \in \langle x \rangle$. By \cite[Lemma 3.3]{m2020e} $y \in C_G(u)$.

If $[z,y]^7=1$, then $v=z^{[y^{-1},z][y,z]}$ is an involution centralizing $y$. By \cite[Lemma 3.1]{m2020e} $u=v$. Then $\langle z,x,y \rangle$ is a homomorphic image of $$G(i)=\big \langle z,x,y \mid \rho_{42} \cup \{y^3,u^yu,[x,y],(yz)^6,(z^yz)^7,uv,((zy)^2x)^i \} \big\rangle ,$$ where $i \in \{4,5,6,7\}$. Computations show that $xy$ is in the center of $G(6)$, which is not possible, and the orders of other groups $G(i)$ are not greater than $42=| \langle z,x \rangle|$, therefore $y \in \langle x \rangle$, which contradicts the choice of $y$.

If the order of $[z,y]$ is not $7$, then, by Lemma \ref{2+3-A4}, it divides $3$ or $5$. It follows that $\langle z,x,y \rangle$ is a homomorphic image of $$G(i_1,i_2)=\big \langle z,x,y \mid \rho_{42} \cup \sigma \cup \{(z^yz)^{i_1},((zy)^2x)^{i_2} \} \big \rangle$$ where 
$$ \sigma=\big \{y^3,u^yu,[x,y],(yz)^6,(z \cdot xy)^6,(z \cdot xy^{-1})^6 \big\},$$
$i_1 \in \{3,5\}$ and $i_2 \in \{4,5,6,7\}$. 

Computations show that the index $|G(i_1,i_2):\langle x,z \rangle|$ divides $3$. This index equals $3$ only for groups $G(3,6)$ and $G(5,6)$ which contain an element of order $3$ in the center. It follows that $\langle z,x,y \rangle = \langle x,z \rangle$ and this contradicts the choice of $y$. 

Let now $v$ be an involution in $C_G(x)$. Then $(uv)^6=1$. 

If $|uv|=3$, then by item (1) $uv=x^{\pm 1}$, and therefore $v=ux^{\pm 1}$ is an involution, which is a contradiction. So we may assume that a dihedral subgroup $\langle u,v \rangle \subseteq C_G(x)$ has no elements of order 3 and we may choose $v$ so that $[u,v]=1$. Then $v=u$ by \cite[Lemma 3.1]{m2020e}.

\vspace{1mm}

{\bf (2)} $C_{\Delta}(u)=\{u\}$.

Assume that there is an involution $t \in C_{\Delta}(u)$ such that $t \not = u$ and let $w=[x,t]$. Note that $\langle t,x \rangle \subseteq C_G(u)$, and therefore $w^3=1$ by Lemma \ref{2+3-A4}. By (1) we have $w \not = 1$. So $\langle x,w \rangle$ is either elementary abelian or isomorphic to the extraspecial group $3^{1+2}$ of exponent 3 and order $3^3$. In any case either there is an element $y\in \Gamma_3 \backslash \langle x \rangle$ centralizing $x$, which is a contradiction by item (1), or $(tx)^2=1$. In the last case $\langle t,z,x \rangle$ is a homomorphic image of a group $$G(i_1,i_2)=\big \langle z,x,t \mid \rho_{42} \cup \{ t^2,(ut)^2,(tx)^2,(t^zx)^6,(z^tx)^6,(tz)^{i_1},(tzx)^{i_2} \} \big \rangle,$$ where $i_1,i_2 \in \{4,5,6,7\}$. Computations show that $G(7,5) \simeq 
G(7,7) \simeq S_7$, $G(6,7) \simeq L_3(2):2$, which is not possible and the order of $G(i_1,i_2)$ divides $12$ for other values of the parameters. So we have a contradiction.

\vspace{1mm} 

{\bf (3)} If $t \in \Gamma_2 \backslash \Delta$, then $[u,t]=1$.

We consider the various possibilities for $\langle t,x \rangle$ using \cite[Lemma 2.1]{m2020e}. 

Assume $(tx)^2=1$. Note that $t$ and $u$ are not conjugated, therefore $|ut|$ is even. If $|ut|=4$ then $(ut)^2 \in C_{\Delta}(u)=\{u\}$, which is not possible. It follows that $(ut)^6=1$ and $(u \cdot tx)^6=1$ ($t$ and $tx$ are conjugated in $\langle t,x \rangle \simeq S_3$). Therefore $\langle u,t,x \rangle$ is a homomorphic image of $$\big \langle u,t,x \; \big | \; \{u^2,t^2,x^3,[u,x],(tx)^2,(ut)^6,(utx)^6 \} \big \rangle \simeq S_3 \times S_3.$$ From item (1) it follows that $(ut)^2=1$.

{\bf (3.1)} Assume $(tx)^3=1$. Then $\langle t,x \rangle \simeq A_4$. By \cite[Lemma 4.2]{m2020e} either $G$ contains a subgroup isomorphic to $S_5$, or $O_2(\langle t,x \rangle) \subseteq O_2(G)$. In other words, we may assume that $(ut)^4=(ut^x)^4=1$. 

If $|ut|=4$ then $(ut)^2 \in C_{\Delta}(u)=\{u\}$, which is a contradiction. Therefore $(ut)^2=1$.

If $(tx)^4=1$, then $\langle t,x \rangle \simeq S_4$. So $(tx)^2 \in \Delta$ and $\langle (tx)^2,x \rangle \simeq A_4$, a contradiction with Statement \ref{s:arxiv}.

{\bf (3.2)} Assume $(tx)^5=1$, then $\langle t,x \rangle \simeq A_5$. 

We may identify $t=(1,2)(3,4)$ and $x=(1,3,5)$. Take $v=(1,3)(2,4)$ and $w=(1,3)(2,5)$. Then $\langle x,v \rangle \simeq S_3$ and $\langle w,x \rangle \simeq A_4$. It is already proven that in these cases $v$, $w$ and $w^x$ centralize $u$. Note that $t \in \langle v,w,w^x \rangle \leq C_G(u)$, and in this case we are done.

{\bf (3.3)} Assume $(tx)^7=1$. 

We have $\langle t,x \rangle \simeq L_{2}(7)$, against the hypothesis that $G$ does not possess a finite nonabelian simple subgroup $H$ with $\Gamma_{2}(H) \subseteq \Delta$.

{\bf (3.4)} $tx \in \Gamma_6$.

If $[x,t]$ is of order $5$ or $7$ then $C_{\langle x,t \rangle}(x)$ contains an involution, which is conjugated with $t$. By item (1) it should coincide with $u$ and hence $t \in \Delta$, a contradiction.

If $[x,t] \in \Gamma_4$, then $v=[x,t]^2 \in \Delta$ and $\langle v,x \rangle \simeq A_4$, which is a contradiction by Statement \ref{s:arxiv}.

Therefore $[x,t]^6=1$. If $y=[x,t]^2$, then $\langle x,y \rangle$ is a homomorphic image of an extraspecial group $3^{1+2}$. By item (1) we get from here that $y=1$. So $\langle x,y \rangle$ is a homomorphic image of $K=\langle x,t \mid \{t^2,x^3,(xt)^6,[t,x]^2 \} \rangle \simeq C_2 \times A_4$. If $Z(\langle x,t \rangle)$ contains an involution, then, by item (1), it should coincide with $u$, and so $[t,u]=1$ as claimed. The other possibility $\langle x,t \rangle \simeq A_4$ was already considered.

{\bf (4)} Lemma follows from item (3).

Note that $u$ and $z$ are conjugated. So $u,z \in C_G(t)$ by (3). It follows that the element $uz$ of order $7$ centralizes $t$, a contradiction. 
\end{proof}

\begin{lemma}\label{l:noD10}
If $a \in \Delta$ inverts an element of order $5$, then $\Gamma_{2}(C_G(a))=\{a\}$.
\end{lemma}

\begin{proof}
From \cite[Statement 2]{m2020e} it follows that if $a \in \Delta$ inverts an element of order $5$ and $C_G(a)$ contains an involution $t \not = a$, then $G$ contains a subgroup isomorphic to $A_6, L_3(4)$ or $S_5$, which is contrary to hypotheses. 
\end{proof} 

\begin{lemma}\label{l:333} Let $a\in \Delta$ and $b,c \in \Gamma_{2}$ with $(ab)^3=(bc)^3=1$. Then $(ac)^3=1$.
\end{lemma}

\begin{proof}
Using Lemma \ref{2+3-A4} we obtain that $\langle a,b,c \rangle$ is a homomorphic image of 
$$G(i) =\big \langle a,b,c \; \big | \; \{ a^2,b^2,c^2,(ab)^3,(bc)^3,(ab \cdot c)^6,(ab \cdot a^c)^6,
(ab \cdot a^{cac})^6, (ac)^i \} \big \rangle,$$ where $i \in \{4,5,6,7\}$. Computations show that $|G(i)|$ divides $6$ if $i \not = 6$. 

Let $u=(ac)^3$. In $G(6)$ the order of $bu$ is $8$ and the order of $a^{cb}\cdot c^a$ is $36$. So we consider $\overline{G}=G(6) / \langle (bu)^4, (a^{cb}\cdot c^a)^6 \rangle$. In $\overline{G}$ we have $\langle ab, c^{ac} \rangle \simeq S_4$ which by Lemma \ref{2+3-A4} implies $(ab \cdot c^{ac})^2=1$. So we get to a factor $$\overline{G} / \langle (abc^{ac})^2 \rangle^{\overline{G}} \simeq 3^{1+2}:2,$$ in which the product of two involutions has order 1 or 3 and this proves the lemma.
\end{proof}

We now prove an analogue of Baer-Suzuki theorem for $p=3$.

\begin{lemma}\label{l:bs3}
Assume that $y \in \Gamma_3$ and any two elements from $y^G$ generate a $3$-group. Then $H=\langle y^G \rangle$ is a $3$-group.
\end{lemma}

\begin{proof}
Let $h=y_1\ldots y_n$ be an arbitrary element of $H$ where $y_i \in y^G$. Denote $x= y_1\ldots y_{n-1}$ and $y=y_n$. Using induction on $n$, it is sufficient to prove that if the order of $x$ is $3$, then the order of $xy$ is $3$. By assumption the following relations hold: $\rho = \{ (y^xy)^3, (y^xy^{-1})^3, [y^x,y]^3, (y^{xyx}y)^3  \}$. Therefore $\langle x,y \rangle $ is a homomorphic image of $$K=\big \langle x,y \; \big | \; \rho \cup \{x^3,y^3\} \big\rangle$$ and computations show that $K$ is a finite group of order $3^9$.
\end{proof}

\begin{lemma}\label{l:D10}
If $a \in \Delta$, then $a$ inverts no elements of order $5$.
\end{lemma}

\begin{proof} Assume the contrary. By Lemma \ref{l:noD10} $C_G(a)$, contains the unique involution $a$, hence for every $b \in \Gamma_2$ the order of $ab$ is odd and $\Delta=\Gamma_2=a^G$.

Consider a graph $\Gamma=\big (\Delta,E \big)$ with vertices $\Delta$ and edges $E=\{(a,b) \mid a,b \in \Delta, ab \in \Gamma_{3} \}$. Let $\Delta_b$ be its connected component, passing through a vertex $b \in \Delta$, then, by Lemma $\ref{l:333}$, $\Delta_b$ is a complete graph.

Assume that $\Delta_a$ has a vertex $b \not =a$. Let $x=ab$. Note that $C_G(x)$ is a $3$-subgroup. Indeed, $K = \langle a,b,c \mid \{a^2,b^2,c^2,[c,ab],(ab)^r,(ac)^s \} \rangle \simeq D_{2m}$, where $m=\mathrm{GCD}(r,s)$ and $a=b$ in $K$.

Let $c \in \Delta$ and $d \in \Delta_c$. By assumptions $(dx)^6=1$ and $(dd^x)^p=1$. By \cite[Lemma 2.1]{m2020e} if $p \in \{5,7\}$, then $C_G(x)$ contains an involution, which is not possible. Therefore $p=3$ and $d^x \in \Delta_c$. So $x$ normalizes each connected component $\Delta_b$ of the graph $\Gamma$.

Let $y \in x^G$. Then $y$ normalizes $\Delta_a$. Therefore, $\langle x,x^y \rangle$ is a $3$-subgroup. Computations show that $\langle x,y \mid \{ x^3,y^3,(xx^y)^3,[x,y]^3,(xy)^i \}\rangle$ is a $3$-group for $i \in \{4,5,6,7\}$. Hence $\langle x,y \rangle$ is a $3$-subgroup. By Lemma \ref{l:bs3} $x \in O_3(G)$.

Our proof now uses only the fact that the product of any two involutions is odd, which can be written as an identity, and therefore it is a property that is preserved in homomorphic images of $G$.

Consider $\overline{G}=G / O_3(G)$. If $\mu(\overline{G}) = \{3,4,5,7\}$, then $\overline{G} \simeq L_3(4)$ by \cite{l34e}, and $G$ is locally finite, a contradiction with Lemma \ref{l:nonlf}.  

Assume $\mu(\overline{G}) = \{4,5,7\}$. Then $C_{\overline{G}}(\overline{a})$ is a $2$-group of exponent $4$, containing the unique involtuion. Every infinite locally finite group contains an infinite abelian subgroup \cite{hallkulatilatika,kargapolove}. It follows that
$C_{\overline{G}}(\overline{a})$ is finite. By Shunkov's theorem \cite{shu1972e} $\overline{G}$ is locally finite. By Schmidt's theorem, $G$ is locally finite, a contradiction with Lemma \ref{l:nonlf}.  

Hence $\mu(\overline{G})=\mu(G)$ and we may assume that $\Gamma$ is an empty graph. Let $q \in \Gamma_3$ such that $[q,a]=1$ and take $b \in \Gamma_2$ such that $a \not = b$, then $bb^q \in \Gamma_5$, $qb \in \Gamma_6$ and there is an involution in $\langle q,b \rangle$, which centralizes $q$. It follows that $a \in \langle q,b \rangle$ and $(ab)^5=1$. 

Let $K= \langle \Delta \rangle$,  $r \in \Gamma_4$ and note that by \cite[Lemmas 12 and 13]{lmmj2014e} $Kr \subseteq \Gamma_4$. Therefore both $K$ and $G/K$ have an involution, and no elements of order $4$. 

By hypothesis there is an element $x$ of order $3$ such that $a$ and $x$ do not commute. Then $\langle a,x \rangle \simeq F_{150}$, and $u=a^{xa(x^{-1}a)^2} \in C_G(x)$. There are no elements of order 4 in $K$ and in $C=C_K(u)$, therefore $C \simeq O_3(C) \times \langle u \rangle$, and $O_3(C)$ is infinite by Lemma \ref{l:shunk}. Let $y \in O_3(C)$ such that $\langle y \rangle \not = \langle x \rangle$ and $[x,y]=1$. Let $z=xy$. Select involutions $v = a^{ya(y^{-1}a)^2}  \in C_G(y)$ and $w = a^{za(z^{-1}a)^2} \in C_G(z)$ in groups $\langle a,y \rangle$ and $\langle a,z \rangle$ correspondingly. Having $u,v \in C_G(y)$ and $u,w \in C_G(z)$ we obtain relations $u=v=w$. It follows that $L=\langle a,x,y \rangle$ is a homomorphic image of $$K=\big \langle a,x,y \; \big | \; \kappa \,\big \rangle,$$ where $$\kappa=\big \{a^2,x^3,y^3,(ax)^6,[a,x]^5,(ay)^6,[a,y]^5,[x,y],(az)^6,[a,z]^5,uv,uw \big\}.$$ Computations show that $K$ is finite and has an element $axayx$ of order $15$. Therefore $L$ is a homomorphic image of $K(p)=K/\langle (axayx)^p \rangle$, where $p \in \{3,5\}$. Computations show that $K(3)$ has no elements of order $5$, and $K(5) \simeq F_{150}$, hence $\langle x \rangle = \langle y \rangle$: a contradiction.
\end{proof}

\begin{lemma}\label{l:3and3} Let $a \in \Delta$. If $x \in \Gamma_3$ and $x^{a}=x^{-1}$, then $x \in O_3(G)$.
\end{lemma}

\begin{proof} Let $b=ax$. By Lemma \ref{l:bs3} it is sufficient to prove that for every element $y$ of order $3$ the order of $xy$ divides $3$. By the assumptions we have the following set of relations $$\rho=\big \{a^2,b^2,(ab)^3,(ay)^6,(a^ya)^3,(by)^6,(b^yb)^3 \big \}.$$
	
Let $g=yy^a$. Then $g^3=1$ and $[g^a,g]=1$. By Lemmas \ref{2+3-A4},\ref{l:noF42} and \ref{l:D10} we obtain the relations $$\sigma=\big \{(bg)^6,(b^ag)^6,(b^gb)^3 \big \}.$$ By Lemma \ref{l:333} $(ba^g)^3=(ab^g)^3=1$. The group $\langle a,b,g \rangle$ is a homomorphic image of $$G(i,j)=\big \langle a,b,g \; \big | \; \rho \cup \sigma \cup \{g^3,[g^a,g],(ba^g)^3,(ab^g)^3,(abg)^i,(bag)^j \} \,\big \rangle,$$ where $i,j \in \{4,5,6,7	\}$. Computations show that $G(i,j)$ is a finite group, whose order divides $3^5 \cdot 2$. Hence, $(abg)^3=(bag)^3=1$, similarly, if $h=yy^b$, then $(abh)^3=(bah)^3=1$.  Define the set $\tau$ as
$$ \big \{y^3,(ay)^6,(a^ya)^3,(by)^6,(b^yb)^3,(ba^y)^3,(ab^y)^3,(abg)^3,(bag)^3,(abh)^3,(bah)^3 \big \}, $$
then the group $\langle a,b,y \rangle $ is a homomorphic image of  $$G(i)=\big \langle a,b,y \mid
\rho \cup \tau \cup \{(aby)^i \} \, \big\rangle,$$ where $i \in \{4,5,6,7 \}$.
Computations show that $G(i)$ is a finite group whose order divides $3^7 \cdot 2$, and hence the lemma is proved. 
\end{proof}

\begin{lemma}\label{l:O3} Assume $O_3(G) \not =1 $. Let $\bar{S}$ be a nontrivial 2-subgroup of $\bar{G}=G/O_3(G)$, let $S$ be its full preimage in $G$ and $\Gamma_2(S) \subseteq \Delta(G)$. Then $N_{\bar{G}}(\bar{S})$ has no elements of order $5$ or $7$.
\end{lemma}

\begin{proof} Assume that $N_{\bar{G}}(\bar{S})$ contains an element $\bar{x}$ of order 5 or 7. Then its preimage $x$ acts fixed point freely on $O_3(G)\bar{S}=S$. By \cite{mamontov2013e} $S$ is locally finite. By Schmidt's theorem $\langle S,x \rangle$ is locally finite. Take $a \in \Gamma_{3}(S)$  and $b \in S$ of order dividing $4$. Then $\langle a,b,x \rangle$ is a finite Frobenius group such that $\langle x \rangle$ is its complement and $\langle a,b \rangle$ is in its kernel, which is nilpotent. It follows that
$[a,b]=1$ and $S=O_3(G) \times S'$, where $S'$ is a nontrivial 2-subgroup. 

Let $c$ be an involution in $S'$ and let $y \in G$. Both $c$ and $c^y$ centralize a Sylow 3-subgroup $O_3(G) \not = 1$ of $S$, hence $\langle c,c^y \rangle$ has no elements of orders $5$ and $7$. We have $c \in \Delta$ and hence, by Lemma \ref{l:3and3}, any element of order $3$ inverted by $c$ is in $O_3(G)$. At the same time, $c$ centralizes $O_3(G)$, so  $\langle c,c^y \rangle$ is a $2$-group. By an analog of Baer-Suzuki theorem for $p=2$ (see \cite{bs2014e}) $c \in O_2(G)$, a contradiction by Statement \ref{s:o2}.
\end{proof}

{\it Proof of Statement \ref{fsg}.}

Assume the contrary and take $a \in \Delta$ and $x \in \Gamma_3$. By Lemmas \ref{l:noF42} and \ref{l:noD10} we have $(aa^x)^3=1$. 

Assume first that for all $x \in \Gamma_3$ we have $a=a^x$. Then any two elements of $a^G$ are in the centralizer of an element of order $3$ and generate $2$-subgroup. Using  an analog of Baer-Suzuki theorem for $p=2$ \cite{bs2014e}, we obtain $a \in O_2(G)$, a contradiction with Statement \ref{s:o2}.

Choose $x \in \Gamma_3$ such that $y=aa^x \in \Gamma_3$. Then $y \in O_3(G)$ by Lemma \ref{l:3and3}. Consider $\overline{G}=G/O_3(G)$. 

If $\mu(\overline{G})=\{3,4,5,7\}$, then $\overline{G} \simeq L_3(4)$ by \cite{l34e}, and $G$ is locally finite, a contradiction with Lemma \ref{l:nonlf}.  

Assume that $\omega(\overline{G}) = \omega (G)$. Then $\bar{a}$ commutes with $\Gamma_3(\overline{G}) \not =1$ and, as shown above, $\bar{a} \in O_2(\overline{G})$, a contradiction with Lemma \ref{l:O3}.

Finally, assume that  $\mu(\overline{G}) = \{4,5,7\}$. We denote $H=\overline{G}$, we omit the bars and we split the proof in several steps.

\vspace{1mm}

{\bf (1)}  All involutions of $H$ are conjugated. 

By the properties of dihedral subgroups it is sufficient to prove that if $b \in \Gamma_2(H)$ and $[a,b]=1$, then $a$ and $b$ are conjugated. By Lemma \ref{l:O3}, $\langle \Delta \rangle$ is not a 2-group. By an analog of Baer-Suzuki theorem for $p=2$ \cite{bs2014e}, there is an involution $c$ such that the order of $ac$ is $5$ or $7$. Orders of $bc^a$ and $a^{ca}b$ are even. Therefore $\langle a,b,c \rangle$ 
is a homomorphic image of $$J_p(i)= \big \langle a,b,c \mid  a^2,b^2,c^2,[a,b], (bc^a)^4, (a^{ca}b)^4, (ac)^p, (abc)^i  \big \rangle,$$ with $p \in \{5,7\}$ and $i \in \{4,5,7\}$. Computations show that $J_5(5) \simeq 2^4:D_{10}$, $J_7(7) \simeq 2^6:D_{14}$, $|J_p(i)| $ divides $4$  for the other parameters, moreover $b \in \Delta(J_5(5))$ or $\Delta(J_7(7))$ correspondingly. So we obtain a contradiction by Lemma \ref{l:O3}.

\vspace{1mm}

{\bf (2)} Let $X=\{x \mid x\in H,\, a^x \in C_H(a) \}$. If $x \in X$, then $x^4=1$.

Assume the statement is not true and let 
$$\rho(i_1,i_2,i_3,i_4)=\{ (ax)^{i_1},(ax^2)^{i_2},[a,x^2]^{i_3},((x^2x^a)^2(x^ax^2)^{-1})^{i_4}\},$$ then $\langle a,x \rangle$ is a homomorphic image of $$J_p=J_p(i_1,i_2,i_3,i_4)=\big \langle a,x \; \big | \; \{a^2,x^p,[a,a^x]\} \cup \rho(i_1,i_2,i_3,i_4) \, \big \rangle,$$ where $p\in\{5,7\}$ and $i_1,i_2,i_3,i_4 \in \{4,5,7\}$. Computations show that we have $|J_p| \in \{1,2,p\}$ except the following two cases: $J_5(5,5,4,5)\simeq 2^4:5$ and $J_7(7,7,4,7) \simeq V:7$, where $|V|=2^{12}$, which are not possible by Lemma \ref{l:O3}.

\vspace{1mm}

{\bf (3)} $\Gamma_2(X \setminus C_H(a))=\emptyset$. In particular $x^2 \in C_H(a)$ for every $x \in X$.

Assume $x\in \Gamma_2(X\setminus C_H(a))$, then $aa^x$ is an involution in $C_H(a)$. By definition of $X$ and step (1), there is $y \in X$ such that $a^y=aa^x$.
Since $a^{yx}=aa^x \in C_H(a)$ we have $yx \in X$. From step (2) it follows that $(xy)^4=1=(ayx)^4$ and $(a(yx)^2)^4=1$. Therefore $\langle a,x,y \rangle$ is a homomorphic image of $$T(i_1,i_2) = \big \langle a,x,y \; \big | \; \tau \cup \{ (xy^2)^{i_1},(axyay^2)^{i_2} \}\, \big\rangle,$$
where $$\tau=\{ a^2,x^2,y^4,(ax)^4,(xy)^4,(ayx)^4,(a(yx)^2)^4,a^yaa^x \}$$ and $i_1,i_2 \in \{4,5,7\}$. Computations show that all such groups are finite, and in all of them $a=1$, a contradiction.

\vspace{1mm}

{\bf (4)} If $x,y \in \Gamma_2(H)$ and $(xy)^4=1$, then $(xy)^2=1$. In particular, $J=\langle \Gamma_2(C_H(a)) \rangle$ is elementary abelian.

Let $(xy)^4=1$, then $x^y \in C_H(x)$ and hence $y \in X(x)$. By step (3) we have $y \in C_H(x)$.

\vspace{1mm}

{\bf (5)} There is a contradiction.

Let $x$ be an arbitrary involution commuting with $a$. By step (1) $x=a^g$ and $ax=a^h$ for $g,h \in X$. By step (4) $g,h \in N_H(J)$. By Lemma \ref{l:O3} $N_H(J)$ is a group of exponent 4, hence it is locally finite \cite{sane}. Therefore $L=\langle a,g,h \rangle$ is finite and nilpotent and $$a^h=ax=aa^g=[a,g] \in [L,a],$$ which is impossible (unless $a=1=x$). A contradiction. \hfill $\square$

\end{section}

\begin{section}{Symmetric subgroups}
	
By Statement \ref{fsg}, $G$ has a subgroup $H$ isomophic to $A_5$ or $L_2(7)$ such that $\Gamma_2(H) \subseteq \Delta$.
 
Our strategy of proof is to use the following ``path'' (see {\sc Figure 1}): on each step assume that $G$ contains a subgroup $H$, isomorphic to a vertex label, and deduce that $G$ contains one of the subgroups connected with that vertex by a path, or obtain a contradiction.

\Large
\begin{center}
\begin{tikzcd} 
		& \mathrm{L_{3}(2)} \arrow[d] & \mathrm{A_{5}} \arrow[d] \arrow[rd]  &  & \\ 
		& \mathrm{L_{3}(4)} & \arrow[l] \mathrm{A_{6}} \arrow[d] \arrow[dr] &  \mathrm{S_{5}} \arrow[dl]  \arrow[d]  &    \\
		&& \mathbf{A_{7}} & \mathrm{S_{6}}&\\
\end{tikzcd} 
\large
\\ \vspace{-6mm} {\sc Figure 1} \\
\end{center}
\normalsize

The first step is to consider the case $H \simeq A_5$, where we additionally assume that $\Gamma_2(H) \subseteq \Delta$. In this case there is $a \in \Delta$ which inverts an element of order $5$ and $C_H(a)$ contains an involution $t \not = a$. By Lemma \ref{l:noD10} we deduce that $G$ has a subgroup isomorphic to $A_6$, $L_3(4)$, or $S_5$. 
In this paragraph we consider other symmetric and alternating groups. 

The proof of the followig lemma is naturally obtained from the proof of \cite[Lemma 14]{lmmj2014e}.

\begin{lemma}\label{l:s4} Let $H$ be a subgroup of  $G$ isomorphic to $S_4$ and $V=O_2(H)$. Let $c \in \Gamma_{3}(H)$, $s\in \Gamma_{2}(H)$ such $c^{s}=c^{-1}$ and $v \in H$ such that $\langle v \rangle =C_V(s)$. Set $V_1= \langle s,v \rangle$, and let $S= VV_1$ be a Sylow 2-subgroup of $H$. Then one of the following holds:
\begin{enumerate}
\item $C=C_G(V)=V$, $S$ is a Sylow $2$-subgroup of $G$, and
$N=C_G(v)$ is an extension of elementary abelian $3$-group $R$ by $S$, and $[R,s]=1$.
\item $H$ normalizes nontrivial cyclic subgroup. Moreover, if $\langle V,c \rangle$ is contained in a subgroup, isomorphic to $A_5$, then either $H$ normalizes a cyclic subgroup of order $2$, or $\langle V,c \rangle$ is contained in a subgroup isomorphic to $A_7$.
\item There exists elementary abelian subgroup $W$ of order $4$ in $C$ such that $W \not \leq H$, $H \leq N_G(W)$ and $c$ acts on $W$ fixed point freely.
\end{enumerate}
\end{lemma}

\begin{lemma}\label{l:2and5} Assume $H=\langle x,y,t \mid \sigma \rangle$, where $$\sigma=\{ x^{2},y^{2},t^{2}, (xt)^{2}, (xy)^{5}\}.$$ If $(yt)^{7} \not =1$, then $7 \not \in \omega(H)$ and either $H$ is isomorphic to $S_5$ or $S_6$, or one of the following sets of relations holds:

$$\tau_1=\{(yt)^5,(xyt)^6,((xyt)^2y)^6,(x(yt)^2)^5 \},$$
$$\tau_2=\{ (yt)^6,(xyt)^5,((xyt)^2y)^6,(x(yt)^2)^5 \},$$
$$\tau_3=\{ (yt)^4,(xyt)^5,((xyt)^2y)^4,(x(yt)^2)^4 \},$$ $$\tau_4=\{ (yt)^5,(xyt)^4,((xyt)^2y)^4,(x(yt)^2)^4 \}.$$

Moreover $$\langle x,y,t \mid \sigma \cup \tau_{1} \rangle \simeq \langle x,y,t \mid \tau(6) \cup \tau_{2} \rangle \simeq 2^4:A_5,$$ while $$\langle x,y,t \mid \sigma \cup \tau_{3} \rangle \simeq \langle x,y,t \mid \sigma \cup \tau_{4} \rangle \simeq 2^{4}: D_{10}. $$ 
\end{lemma}

\begin{proof} The group $\langle x,y,t \rangle$ is a homomorphic image of one of the following groups:
 $$G(i_1,i_2,i_3,i_4)= \big \langle x,y,t \; \big | \;
\sigma \cup \{(yt)^{i_1},(xyt)^{i_2},((xyt)^2y)^{i_3},(x(yt)^2)^{i_4} \}
\big \rangle , $$ with $i_1 \in \{4,5,6\}$ and $i_2,i_3,i_4 \in \{ 4,5,6,7 \}.$ Computations show that $$G(4,6,5,6) \simeq G(6,4,5,6) \simeq S_5$$ and $$G(6,6,5,4) \simeq S_6,$$ while other nontrivial cases are listed in statement of the lemma. 
\end{proof}

\begin{lemma}\label{l:a6} If $G$ contains a subgroup $H$ isomorphic to $A_6$, then $G$ contains a subgroup isomorphic to $S_6$, $A_7$, or $L_3(4)$.
\end{lemma}

\begin{proof} We identify $H$ with $A_6$ and we denote $a=(1,2,3)$, $b=(1,2,4)$, $c=(1,2,5)$, and $d=(1,2,6)$. Then $a,b,c,d$ generate $A_6$ and satisfy the following set of relations:
$$\alpha=\{a^3,b^3,c^3,d^3,(ab)^2,(ac)^2,(ad)^2,(bc)^2,(bd)^2,(cd)^2 \},$$ furthermore $\alpha$ gives a presentation for $A_6$. Also denote $x=ac=(1,5)(2,3)$, $y=cd=(1,6)(2,5)$, and $z=ab=(1,5)(2,3)$.

Set $V= \langle x, x^a \rangle$, $s=x^{abcad}=(2,5)(4,6)$, then
$N_H(V)= V\langle c,s \rangle \simeq VS_3 \simeq S_4$ and let $V_1$ as in Lemma \ref{l:s4}. If $C_G(V)=V$, then by Lemma \ref{l:s4} $C_G(V_1)>V_1$. So, since all involutions of $H$ are conjugated, we may assume that $C_G(V)>V$.

By Lemma \ref{l:s4} either $G$ contains a subgroup isomorphic to $A_7$, or one of the following holds:

\vspace{1mm}

{\bf (1)} There is an involution $t \in \Gamma_2(G)\backslash H$, which centralizes $N_H(V)$, or in terms of relations $$\rho= \big \{ t^2,[c,t],[x,t],[x^a,t],[s,t]  \big \}.$$ Note that $xz \in \Gamma_3$, therefore the order of $tz$ is not $7$ by \cite[Lemma 2.2]{m2020e}. We prove that $\langle H,t \rangle \simeq S_6$, moreover $t \simeq (4,6)$ if we use natural embedding for $H$. The group $\langle H,t \rangle$ is a homomorphic image of  $G(j,i_1,i_2,i_3,i_4,i_5,i_6)$ defined as $$\big \langle a,b,c,d,t  \; \big | \; \alpha \cup \rho \cup \{(tz)^j,(at)^{i_1},(bt)^{i_2},(dt)^{i_3},(a^bdt)^{i_4},(abcdt)^{i_5},(a^bcdt)^{i_6}\}
\big \rangle,$$ where $x=ac$ and $s=x^{abcad}$. Computations show that for all $j\in \{4,5,6\}$ and $i_1,i_2,i_3,i_4,i_5,i_6 \in \{ 4,5,6,7 \}$ the order of $G(j,i_1,i_2,i_3,i_4,i_5,i_6)$ is finite and greater than $|A_6|$ only for 
$G(6,6,4,4,6,6,6)\simeq S_6$. Calculations in $S_6$ show that $t \simeq (4,6)$ as desired.

\vspace{1mm}

{\bf (2)} There is an involution $t \in \Gamma_2(G)\backslash H$ such that relations $$\sigma= \big \{t^2,[x,t],[x^a,t],(tc)^3 \big \}$$ hold. We prove that $\langle H,t \rangle \simeq L_3(4)$.

Note that $H_1=\langle x,y,t \rangle$ and $H_2=\langle
x^a,z,t \rangle$ satisfy the conditions of Lemma \ref{l:2and5}, so we plan to use the corresponding relations that define these subgroups:
$$\psi_1(i_1,i_2,i_3,i_4)=\{(yt)^{i_1},(xyt)^{i_2},((xyt)^2y)^{i_3},(x(yt)^2)^{i_4} \},$$
$$\psi_2(j_1,j_2,j_3,j_4)=\{(zt)^{j_1},(x^azt)^{j_2},((x^azt)^2z)^{j_3},(x^a(zt)^2)^{j_4} \}.$$

Therefore $\langle H,t \rangle$ is a homomorphic image of 
$$G(i_1,i_2,i_3,i_4,j_1,j_2,j_3,j_4) = \langle a,b,c,d,t \; \big | \; \alpha \cup \sigma \cup \psi_1 \cup \psi_2 \rangle. $$

By Lemma \ref{l:2and5} we may assume that $(i_1,i_2) \not = (6,6) \not = (j_1,j_2)$, and $i_1,i_2,i_3,i_4,j_1,j_2,j_3,j_4 \in \{4,5,6\}$. Computations show that for all such parameters index of $\langle a,b,c,d \rangle$ is finite; moreover, $\langle H,t \rangle$ is a homomorphic image of  $ L_3(4)$, $2.L_3(4))$ or $2^5:A_6$.
\end{proof}

We remark that $L_3(4)$ is defined by the relations $$\alpha \cup \sigma \cup \{ (at)^3, (bt)^5,(ct)^3,(dt)^3,(abt)^5 \}.$$

\begin{lemma}\label{l:s5} If $G$ contains a subgroup $H$ isomorphic to $S_5$, then $G$ contains a subgroup isomorphic to $S_6$ or $A_7$.
\end{lemma}

\begin{proof} Assume the contrary. We identify $H$ with $S_5$ and $a=(1,2)$, $b=(2,3)$, $c=(3,4)$, $d=(4,5)$. Then $a,b,c,d$ generate $S_5$ and satisfy the following identities: $$\beta=\big \{ a^2, b^2, c^2, d^2, (ab)^3, (ac)^2, (ad)^2, (bc)^3, (bd)^2, (cd)^3 \big \},$$
which provide a presentation for $S_5$. We also denote $x=ab$, $u=ac$, $v=u^x$, and $V=\langle u,v \rangle$. Then $N_H(V)=V \langle a,b \rangle \simeq VS_3 \simeq S_4$.
	
By Lemma \ref{l:s4} we need to consider three cases:

\vspace{1mm}	

{\bf (1)} There is $t \in \Gamma_3$ such that the relations of the following set $$\rho=\big \{ t^{3}, [t,a], [t,c], (t(abc)^2)^2 \big \}$$ hold in $G$. 

We consider the subgroup $K=\langle a,b,c \rangle \simeq S_{4}$ of $H$, then $K=\langle a,b,c \mid \beta^{\ast} \rangle$
where $\beta^{\ast}=\{ a^2, b^2, c^2, (ab)^3, (ac)^2, (bc)^3\}$. In this case $\langle K,t \rangle$ is a homomorphic image of $$J= J(i_1,i_2)=\big \langle a,b,c,t \; \big | \; \beta^{\ast} \cup \rho \cup \{(bt)^{i_1}, (abt)^{i_2}\} \, \big \rangle.$$
Computations show that $J(5,7) \simeq A_7$ and $|J(i_{1},i_{2})| \leq 24$ if $(i_1,i_2) \not =(5,7)$.

\vspace{1mm}

{\bf (2)} There is $t\in \Gamma_2$ such that $[t,V]=[t,x]=1$.
	
In this case $\langle H,t \rangle$ is a homomorphic image of the group $$G(i_1,i_2,i_3,i_4,i_5) = \big \langle a,b,c,d,t \mid \beta \cup \gamma(i_1,i_2,i_3,i_4,i_5) \big \rangle,$$ where $\gamma(i_1,i_2,i_3,i_4,i_5)$ is the set $$\{t^2,(tu)^2,(tv)^2,tt^x,(abct)^{i_1},(adt)^{i_2},(abcdt)^{i_3},((dt)^2c)^{i_4},(bcdt)^{i_5} \}$$ and $i_1,i_2,i_3,i_4,i_5 \in \{4,5,6,7\}$. Computations show that if $|G|>120$, then $$G( 4, 6, 6, 4, 5) \simeq S_{6} \;\;\;\, \mbox{ or } \;\; \; G(4, 6, 6, 5, 7) \simeq A_{8}.$$

\vspace{1mm}
	
{\bf (3)} There is $t\in \Gamma_2$ such that $[t,V]=1$ and $(tx)^3=1$, i.e. the following identities hold $$\sigma=\big \{t^2,(tu)^2,(tv)^2,(tx)^3 \big \}.$$
In this case $\langle H,t \rangle$ is a homomorphic image of $$\widehat{G}=G(j_1,j_2,i_1,i_2,i_3,i_4,i_5,i_6,i_7)$$ defined by $$ \big \langle a,b,c,d,t \; \big | \; \beta \cup \sigma \cup \delta_{1}(j_1,j_2) \cup \delta_{2}(i_1,i_2,i_3,i_4,i_5,i_6,i_7) \big \rangle,$$ where  $$\delta_{2}=\{(bt)^{i_1},(dt)^{i_2},(abct)^{i_3},(bdt)^{i_4},((bcd)^2t)^{i_5},(adt)^{i_6},(abcdt)^{i_7}\}$$ with $i_1,i_2,i_3,i_4,i_5,i_6,i_7 \in \{4,5,6,7\}$ and $$ \delta_{1}=\{(at)^{j_1},(ct)^{j_2} \}$$ with $j_{1},j_{2} \in \{4,6\}$, since $a,c \in C_G(u)$.

Computations show that $|\widehat{G}:\widehat{H}| \in \{1,2^4,2^5,2^{10},2^{11},2^{16}\}$ (and in cases where $|\widehat{G}:\widehat{H}| > 1$, then $\widehat{H} \simeq S_{5}$). So it follows that this case cannot happen.
\end{proof}

\begin{lemma}\label{l:s6} There are no subgroups isomorhpic to $S_6$ in $G$. 
\end{lemma}

\begin{proof}
Assume the contrary and identify a subgroup $H$ with $S_6$. Denote $a=(12)$, $b=(23)$, $c=(34)$,	$d=(45)$, $e=(56)$. Then the following identities hold:
$$\kappa=\kappa_{1} \cup \kappa_{2}$$ where $$\kappa_{1}=\{a^2,b^2,c^2,d^2,e^2\}$$ and $$\kappa_{2}=\{ (ab)^3,(ac)^2,(ad)^2,(ae)^2,(bc)^3,(bd)^2,(be)^2,
(cd)^3,(ce)^2,(de)^3 \}.$$

Analyzing the structure of $C_G(e)$, it is shown in \cite[Lemma 18]{lmmj2014e}
that it is possible to choose $t \in C_G(e) \backslash C_H(e)$ so that one of the following sets of relations hold:
$$\pi_1=\{t^3,(at)^2,(bt)^2,(ct)^2,[ab,t]\},$$
$$\pi_2=\{t^2,[a,t],[b,t],[c,t]\},$$
$$\pi_3=\{t^2,[a,t]e,[b,t]e,[c,t]e \},$$
$$\pi_4=\{t^2e,[a,t],[b,t],[c,t] \},$$
$$\pi_5=\{t^2e,[a,t]e,[b,t]e,[c,t] \},$$
$$\pi_6=\{(at)^2,(bt)^4,(ct)^4,(abt)^3,(act)^2,(bct)^3\}.$$

Also denote $$\chi=\{[e,t],(td)^{i_1},(tcd)^{i_2},(ted)^{i_3},(tad)^{i_4},(tbd)^{i_5},(tabd)^{i_6},(tbcd)^{i_7} \}.$$ Then $\langle H,t \rangle$ is a homomorphic image of a group $$\langle a,b,c,d,e,t \mid \kappa \cup \pi_{\ell} \cup \chi \rangle, \;\;\;\; \ell \in \{1,2,3,4,5,6\}.$$

Computations show that these groups are finite for all $\ell \in \{1,2,3,4,5,6\}$ and all  $i_1,i_2,i_3,i_4,i_5,i_6 , i_7 \in \{4,5,6,7\}$. Hence the lemma follows.
\end{proof}

In the following lemma we show that there is only one natural way to ``glue'' together two subgroups isomorphic to $A_6$ with a common subgroup isomorphic to $A_5$. In the following with $\sim$ we denote an isomorphism.

\begin{lemma}\label{l:a6+a6} Assume that $a \sim (1,2,3)$, $b \sim (1,2,4)$, $c \sim (1,2,5)$, $d \sim (1,2,6)$ are elements of $G$ generating a subgroup $H \sim A_6$.
Assume that $t \sim (1,2,6')$ is such that $\langle a,b,c,t \rangle$ is another subgroup of $G$ isomorphic to $A_6$ ($t \not \in H$). Then $\langle a,b,c,d,t \rangle \simeq A_7$ and we may assume that $t \sim (1,2,7)$.
\end{lemma}

\begin{proof} Note that
$$a_1=bc^{-1}b^{-1}db^{-1}d^{-1} \simeq (4,5,6), a_2=bc^{-1}b^{-1}tb^{-1}t^{-1} \in C_G(a), $$
$$b_1=ad^{-1}a^{-1}ca^{-1}c^{-1} \simeq (3,6,5), b_2=at^{-1}a^{-1}ca^{-1}c^{-1} \in C_G(b), $$
$$c_1=ad^{-1}a^{-1}ba^{-1}b^{-1} \simeq (3,6,4), c_2=at^{-1}a^{-1}ba^{-1}b^{-1} \in C_G(c),$$
and therefore the following relations hold in $G$: $$\rho_c=\{(a_1a_2)^6,(b_1b_2)^6,(c_1c_2)^6 \}.$$
	
Let
$$\xi=\{ a^3,b^3,c^3, (ab)^2, (bc)^2, (ac)^2\},$$ and	
$$\mu_1=\{ d^3, (ad)^2,(bd)^2,(cd)^2 \},\;\;\; \mu_2=\{ t^3, (at)^2,(bt)^2,(ct)^2 \}.$$
	
We have $\langle a,b,c \mid \xi \rangle \simeq A_5$ and
$$ \big \langle a,b,c,d \mid \xi \cup \mu_1 \big \rangle \simeq A_{6} \simeq \big \langle a,b,c,t \mid \xi \cup \mu_2 \big \rangle. $$  Therefore $\langle a,b,c,d,t \rangle$ is a homomorphic image of 
$$K=K(i_1,i_2,i_3,i_4,i_5)= \big \langle a,b,c,d,t \mid \mu_1\cup \mu_2 \cup \xi_5 \cup \eta \,\big \rangle$$ where $$ \eta=\eta(i_1,i_2,i_3,i_4,i_5)=\{(td)^{i_1},(abcdt)^{i_2},[t,d]^{i_3},(t^{-1}d)^{i_4},(adt)^{i_5} \}.$$
	
Computations show that if $i_1,i_2,i_3,i_4 \in \{4,5,6,7\}$ and $i_5 \in \{4,5,6\}$, then $|K:\langle a,b,c \rangle|$ divides $42$. If $i_5=7$ then $\langle a,b,c,d,t \rangle$ is a homomorphic image of $$K(i_1,i_2,i_3,i_4,7) \big / \big \langle (ad)^tad)^4 \big \rangle^{K(i_1,i_2,i_3,i_4,7) }$$ by \cite[Lemma 2.1]{m2020e}, and computations show that for all $i_1,i_2,i_3,i_4 \in \{4,5,6,7\}$ we have $|K : \langle a,b,c \rangle|=1$.
\end{proof}

\begin{lemma}\label{l:a7} There are no subgroups isomorphic to $A_7$ in $G$.
\end{lemma}

\begin{proof}
Assume the contrary and identify a subgroup $H$ with $A_7$. 

Let $u = (1,2)(3,4)$, $v = (1,3)(2,4)$, $w=(1,2)(6,7)$. Denote $U=\langle u,v \rangle$, $V=\langle u,w \rangle$ and use the same notations for isomorphic images in $A_7$. Then $S=UV$ is a Sylow 2-subgroup in $H$. Note that $$C_{A_7}(U)=U \times \langle (5,6,7) \rangle \mbox{ and } C_{A_7}(V)=V,$$
$$N_{A_7}(U)=S_4 \times \langle (5,6,7) \rangle \mbox{ and } N_{A_7}(V)=S_4,$$ here $x=(1,3,6)(2,4,7) \in N_{A_7}(V)$, $x^v=x^{-1}$, $u^x=(3,4)(6,7)=uw$.

Therefore even though both $U$ and $V$ are normal in the corresponding subgroups of $H$ isomorphic to $S_4$, applying Lemma \ref{l:s4} to them will yield different effect. This observation together with the fact that all involutions of $A_7$ are conjugated is the key to the proof.

We apply Lemma \ref{l:s4} to subgroup $K=\langle v,w,x \rangle \simeq S_4$ such that $O_2(K) = V$ and we consider the following cases.

\vspace{1mm}

{\bf (1)} $C_G(V)=V$, $S$ is a Sylow 2-subgroup of $G$, $C_G(u)$ is an extension of an elementary abelian 3-group $Q$ by $S$ and $[Q,v]=1$.

It is convenient to present $H$ using generators and relations as it appears in case 2 of Lemma \ref{l:s4}, when it is applied to a subgroup $\langle U, (1,2,3) \rangle$ isomorphic to $S_4$. These relations and the corresponding permutations are shown on the upper half of the diagram in {\sc Figure 2}.
\large
\begin{center}
	\begin{tikzcd} 
		&   & d \sim (5,6,7) \arrow[ld,dash, dashed] \arrow[rd,dash, dashed] \arrow[rrd,dash,"5"] \arrow[dd,dash,dashed] &  & &\\ 
		& a\sim (1,2,3) \arrow[rr,dash]  &  & b \sim (1,2,4)  \arrow[r,dash] & c\sim (1,2,5) &\\
		&   & t \sim (5,6,7)' \arrow[lu,dash, dashed] \arrow[ru,dash, dashed] \arrow[rru,dash,"5"] & && \\
	\end{tikzcd} 
	\\ \vspace{-4mm} {\sc Figure 2} \\
\end{center}
\normalsize
In other words we fix elements $a,b,c,d \in G$, which satisfy the following relations:
$$\rho_H=\{a^3,b^3,c^3,d^3,(ab)^2, (bc)^2, (ac)^2, [a,d], [b,d], (dc)^5, (abcd)^7\},$$ such that $H=\langle a,b,c,d \rangle \simeq A_7$ and let $u=[a,b]= (1,2)(3,4)$. Then $d \in Q$ and since $Q$ is infinite by Lemma \ref{l:shunk}, there is $t \in Q$ such that $t \not \in H$.

We may assume that $t$ centralizes $a$ and $b$. Indeed, $P=\langle a,b,t \rangle$ is a finite soluble group and $V \subseteq O_2(P)$. If $z \in Z(O_2(P))$, then $z \in C_G(U)$, and hence $z \in S$; also $z \in C_G(v)$, and hence $z \in V$. Therefore $V=O_2(P)$. If $O_3(P)=1$, then $t \in C_P(V) \subseteq V$, a contradiction. Let $1 \not = f \in O_3(P)$, then use a central element from a 3-subgroup $\langle f,a \rangle$ as $t$.

By Lemma \ref{l:s4} $\langle a,b,c,t \rangle \simeq A_7$, moreover $t$ is ``located'' in this group just like $d$ is ``located'' in $H$, hence the following relations are true:
$$\rho_t=\{t^3,[t,a],[t,b],(tc)^5, (abct)^7\}.$$
Note that $\langle a,b,c,d,t \rangle$ is a homomorphic image of $$G(i_1,i_2,i_3)=\langle a,b,c,d,t \; \big | \; \rho_H \cup \rho_t \cup \{ [d,t], (tdc)^{i_1}, (td^c)^{i_2}, (abcdt)^{i_3}\} \rangle, $$ where $i_1,i_2,i_3 \in \{4,5,6,7\}$. 

Computations show that $G(6,5,7) \simeq 3^6:A_7$, which is not possible, and $|G(i_1,i_2,i_3)| \leq |A_7|$ for the other values of parameters.

\vspace{1mm}

Case {\bf (2)} $C_G(V)>V$.
	
First we write the permutations mentioned above as words of generators: $$x=bac(dc)^2, v=a^{-1}b^{-1}, u=v^a, w=uu^x.$$
We proceed with introducing the notations to follow Lemma \ref{l:a6}: $$a_1=x, c_1=x^{vw}, b_1=x^{abcadc}, d_1=x^{abc^{-1}d^{-1}c}, x_1=a_1c_1, y_1=c_1d_1, z_1=a_1b_1,$$ and also let $$s=uv, r=[c,d].$$

Note that the following identities hold:
$$\sigma=\big \{a_1^3,b_1^3,c_1^3,d_1^3,(a_1b_1)^2,(a_1c_1)^2,(a_1d_1)^2,(b_1c_1)^2,(b_1d_1)^2,(c_1d_1)^2 \big \}.$$
Moreover, $x_1=a_1c_1=u$, $V=\langle u,u^{a_1} \rangle$, $s=uv=(1,4)(2,3)$, and $c_1^s=c_1^{-1}$. By Lemma \ref{l:a6} for $\langle a_1,b_1,c_1,d_1 \rangle \simeq A_6$ there is $t\in G$ such that one of the following subcases holds:

{\bf (2.1)}  $\langle a_1,b_1,c_1,d_1,t \rangle \simeq S_6$, and the conclusion follows from Lemma \ref{l:s6}.

{\bf (2.2)}  $\langle a_1,b_1,c_1,d_1,t \rangle \simeq L_3(4)$.

We note that $a,b \in \langle a_1,b_1,c_1,d_1 \rangle$ and there is an involution $$f \in \langle a_1,b_1,c_1,d_1,t \rangle \simeq L_3(4)$$ such that $[f,O_2(\langle a,b \rangle)]=1$ and $(af)^3=1$. In other words $\langle a,b,f \rangle \simeq 2^4:3$ and $\langle a,f \rangle \simeq A_4$. Then $fa^{-1}, a, d$ satisfy the conditions of \cite[Lemma 4.6]{m2020e}. Moreover, $\langle f,d \rangle $ is contained in the centralizer of an involution $ab$ and so has no elements of orders $5$ or $7$. Therefore $f \in O_2(\langle a,f,d \rangle)$ and $(ff^d)^4=1$.

Denote $e=(a^{-1})^{cbadb}\simeq (1,2,6)$ and $z=ab$. Applying Lemma \ref{l:a6} to $\langle a,b,c,e \rangle $ and $f$ we conclude that the following relations hold:  

$$\tau=\{[f,z],[f,z^a],(af)^3,(bf)^3,(cf)^3,(ef)^5,(aef)^5\}.$$

Computations show that the group
$$\langle a,b,c,d,f \mid \rho_H \cup \tau \cup \{(fd)^{i_1},(afd)^{i_2}, (afd^{-1})^{i_3} \}$$ is trivial for $i_1 \in \{4,5,6\}$ and $i_2,i_3 \in \{4,5,6,7\}$ unless $i_1=i_2=i_3=6$. We conclude that the following relations hold in $G$:
$$\chi=\{(fd)^6,(afd)^6, (afd^{-1})^6 \}.$$
Therefore $\langle a,f,d \rangle$ is a homomorphic image of $$K=\langle a,f,d \mid \{a^3,f^2,d^3,[a,d],(af)^3,(ff^d)^4\}\cup \chi \rangle$$ and computations show that $K$ is a finite group of order $2^{10} \cdot 3^2$. We may choose $f$ to be in $Z(O_2(\langle a,f,d \rangle))$, and assume that $$(fd)^3=(afd)^3=1 \;\;\; \mbox{ or } \;\;\; [f,d]=1.$$

Finally, $\langle a,b,c,d,f \rangle$ is a homomorphic image of $$\big \langle a,b,c,d,f \mid \rho_H \cup \tau \cup \mu_{\ell} \big \rangle, \;\;\;\;\;\; \ell \in \{1,2\},$$ where $\mu_{1}= \{ \{(fd)^3,(afd)^3\}$ and $\mu_{2}= \{ [f,d] \} \}$. Computations show that in both cases $|\langle a,b,c,d,f \rangle : \langle a,b,c,d \rangle|=1$.

{\bf (2.3)}  $\langle a_1,b_1,c_1,t \rangle \simeq A_7$. Let $f \in \langle a_1, b_1,c_1,t \rangle$ be an element that $\langle a_1, b_1,c_1,f \rangle \simeq A_6$ and the following identities hold $$\rho_f = \big \{ f^3,(a_1f)^2,(b_1f)^2, (c_1f)^2 \big \}.$$ By Lemma \ref{l:a6+a6} $\langle a_1, b_1,c_1,d_1,f \rangle \simeq A_7$ and we may assume that $(d_1f)^2=1$.

Let $\beta=\rho_H \cup \rho_f \cup \{(d_1f)^2\}$. Denote $S=\langle a_1, b_1, c_1, d_1 \rangle$. To distinguish isomorphisms we will denote them in the following way $H=\langle a,b,c,d \rangle \simeq_{\uparrow} A_7$ and $\langle S,f \rangle \simeq_{\downarrow} A_7$ so that $$a \simeq_{\uparrow} (1,2,3), b \simeq_{\uparrow} (1,2,4), c \simeq_{\uparrow} (1,2,5), d \simeq_{\uparrow} (5,6,7),$$ and $$a_1 \simeq_{\downarrow} (1,2,3), b_1 \simeq_{\downarrow} (1,2,4), c_1 \simeq_{\downarrow} (1,2,5), d_1 \simeq_{\downarrow} (1,2,6), f \simeq_{\downarrow} (1,2,7).$$ Let $i\simeq_{\uparrow}(1,2)(3,5)=a_1b_1^{(a_1b_1d_1^{-1}b_1^{-1}c_1d_1^{-1}c_1^{-1})} \simeq_{\downarrow} (1,2)(6,7)$. Then $i$ inverts $b_1,d_1,f$ $(\downarrow)$ and $ a,b,c,d $ $(\uparrow)$.

Computations (with enumerating cosets by $\langle a,b,c,d \rangle$) show that a group $$K(w)=\langle a,b,c,d,f \mid \beta \cup \{ w\} \rangle $$ is isomorphic to $PSU(3,5)$ if $w$ is the 4-th power of one of the words in $ W=\{cd,df,cf^{-1},df^{-1}\}$, and $K(w)$ is trivial if $w$ is the 5-th power of a word in $W$. It follows that in $G$ the order of words from $W$ divide $6$ or $7$. Also we have that index of $\langle a,b,c,d \rangle$ in $K((cif)^4)$, $K((dif)^4)$, $K((acif)^4)$, $K((adif)^4)$ is $1$.

We remark that $\langle c,i,f \rangle$ is a homomorphic image (with $ci \rightarrow a, i \rightarrow b, if \rightarrow c$) of  $K=K(j,i_1,i_2,i_3,i_4)$ with presentation $$ \big \langle a,b,c \; \big | \; \{a^2,b^2,c^2,(ac)^j, (abc)^{i_1},(c^{ab}c)^{i_2},(ba^c)^{i_3},(abcac)^{i_4}\} \big \rangle.$$ Also we remark that $i \cdot if \cdot ci \cdot if =fcf \sim cf^{-1}$. So $j,i_3 \in \{6,7\}$, $i_1 \in \{5,6,7\}$ and $i_2,i_4 \in \{4,5,6,7\}$. Computations show that $\langle c,i,f \rangle$ is a homomorphic image of $K(6,6,6,6)$ of order $2^7\cdot 3^3$, and $[c,f]^3=1$. Therefore the following identities hold in $G$: $$\rho_I=\{(cd)^6,(df)^6,(cf^{-1})^6(df^{-1})^6,[c,f]^3,[d,f]^3\}.$$

We then consider a subgroup of $G$ generated by $a,d,f$, and note that its structure is defined by the following relations:
\begin{itemize}
	\item[(i)] $a^3, \; d^3, \; [a,d]$. 
	\item[(ii)] $(af)^3, \; (af^{-1})^6, \; [a,x]$, where $x=(af^{-1})^2$, which hold in $K \simeq_{\downarrow} A_6$.
	\item[(iii)] $(fd)^6, \; (fd^{-1})^6, \; y^3$, where $y=[d,f]$, these relations come from $R_I$.
	\item[(iv)] $(adf)^6, \; (adf^{-1})^6, \; [ad,f]^3$. Since $[a,d]=1$ and $i$ inverts $a$ and $d$, then $i$ inverts $ad$ and we may apply arguments above to $ad$ instead of $d$.
	\item[(v)] $(xd)^3=1$, this relation hold in $\langle a,b,c,d \rangle \simeq_{\uparrow} A_7$ (here we use $x=a_1^{c_1d_1}\simeq_{\downarrow}(3,6,5)=(af^{-1})^2\simeq_{\uparrow}(1,2,3)(4,6,7))$.
	\item[(vi)] $((adf)^2fa)^6$ (this relation is added to eliminate an element of order 12).
\end{itemize}

Let $F$ be a group defined by relations (i) -- (vi). Then computations show that $|F|=2^{12}\cdot 3^4$. 

If $z=f^{-1}af^{-1}dfa^{-1}f$, then $z \in C_F (\langle a,d \rangle ) \setminus \langle a,d \rangle$ and $F/\langle z \rangle^{F} \simeq 3^2$. Therefore we may assume that $z \not =1$ and $z$ is an element of order 3. In the centralizer $C_G(d)$ of an element of order 3 there is a subgroup $\langle b,a,z \rangle$ with $b^3=a^3=z^3=(ba)^2=[a,z]=1$. By \cite[Lemma 4.6]{m2020e} $((ba)(ba)^z)^2=1$.
Computations show that $$ \big |\, \big \langle a,b,c,d,f \mid \beta \cup \rho_I \cup \{((ba)(ba)^z)^2\} \big \rangle : \big \langle a,b,c,d \big \rangle \, \big |= 1$$ and the proof is complete.
\end{proof}

We can summarize the results obtained in

\begin{prop}\label{p:s5}
There are not subgroups isomorphic to $S_5$ in $G$.
\end{prop}

\begin{proof}
	
This is a corollary of Lemmas \ref{l:D10}, \ref{l:a6}--\ref{l:a7}.
\end{proof}

\end{section}

\begin{section}{$L_3(2)$-subgroups}
	
\begin{lemma}\label{l:s3 and comm} 
Let $x,y \in \Gamma_3, a \in \Gamma_2$ be such that $x^a=x^{-1}$ and $y^a=y$. Then $\langle x,y,a \rangle$ is a finite $\{2,3\}$-group and the following relations hold: $$ \chi_{1}(x,y)= \big \{(xy)^6,(yy^x)^3,(xx^y)^3 \big \}.$$ Moreover, $v=[x,y]^2 \in C_G(y)$, and $v^a=v^{-1}$.
\end{lemma}

\begin{proof} Indeed $\langle x,y,a \rangle$ is a homomorphic image of one of the group $$G(i_1,i_2,i_3,i_4,i_5)=\big \langle x,y,a \; \big | \; \sigma \cup \tau(i_1,i_2,i_3,i_4,i_5) \, \big \rangle$$ with  $$\sigma=\{ a^2,x^3,y^3,(ax)^2,[a,y] \} $$ and $$\tau(i_1,i_2,i_3,i_4,i_5)=\{ (xy)^{i_1},(xy^{-1})^{i_2},(axy)^{i_3},(ax^yx)^{i_4},(a(xy)^2)^{i_5} \}.$$

Computations show that $G(5,5,4,4,6) \simeq S_5$, $G(6,6,5,6,5) \simeq L_2(11)$, $G(7,7,5,7,7)\simeq A_7$, which is not possible by Statement \ref{p:s5}; $G(6,6,6,4,6)$ and its homomorphic images $G(6,6,6,6,6)\simeq G(6,6,6,2,6)$ satisfy lemma's conclusion. The order of $G(i_1,i_2,i_3,i_4,i_5)$ is not greater than 6 for other values of parameters $i_1,i_2,i_3,i_4,i_5 \in \{4,5,6,7\}$. 
\end{proof}

\begin{lemma}\label{l:f36} 
Let $t \in \Gamma_4, x \in \Gamma_3$ be such that $\langle t^2 ,x \rangle \simeq S_3$. Then $\langle t,x \rangle$ is isomorphic to $L_3(2)$ or $F_{36}$.
\end{lemma}

\begin{proof} The group  $\langle t,x \rangle$ is a homomorphic image of $$G(i)= \big \langle t,x \mid t^4,x^3,(t^2x)^2, (tx)^i \big \rangle.$$
Computations show that $G(4) \simeq F_{36}$, $G(5) \simeq 1$, $G(6) \simeq S_5$, which is not possible by Statement \ref{p:s5}, and $G(7) \simeq L_2(7)$.
\end{proof}

\begin{lemma}\label{l:f36 and 2}  Assume that $b \in \Gamma_2$ inverts an element $t$ of order $4$ from a subgroup $H$ isomorphic to $F_{36}$. Then $K=\langle b,H \rangle$ is either a finite $\{2,3\}$-group, containing an element $w$ of order 3 such that the following relations hold: $$\chi_2(b,t,w)= \big \{[b,w],(bw^t)^2,(wt^2)^2 \big \},$$ or $K \simeq A_6$.
\end{lemma}

\begin{proof} Let $y \in \Gamma_3(H)$, $a=t^2$ and $$\rho= \big \{t^4,b^2,y^3,(ay)^2,[y,y^t],(tb)^2 \big \}.$$ Note that $K$ is a homomorphic image of $$G(i_1,i_2,i_3,i_4,i_5,i_6, i_7)= \big \langle t,b,y \; \big | \; \rho \cup \kappa \big \rangle$$ where $$ \kappa= \big \{ (aby)^{i_1},(by)^{i_2},(tby)^{i_3},(t(by)^2)^{i_4},(bybty)^{i_5} ,(byy^t)^{i_6},(abyy^t)^{i_7} \big \}.$$ Computations show that $K$ is a homomorphic image of $$S=G(6, 6, 6, 4, 4, 6, 6),$$ and $|S|=2^3 \cdot 3^6$, since other nontrivial possibilities are $G(4,6,5,5,5,5,5) \simeq G(5,5,4,6,5,6,4) \simeq G(5, 5, 6, 5, 6, 4, 6) \simeq G(6,4,5,6,6,5,5) \simeq A_{6}$. 
	
The element $z=(yy^by)^t$ has order dividing 3 in the group $S$ and it satisfies relations $\chi_{2}$. Therefore if $z$ has order 3 in $K$ we can take  $w=z$. If $z=1$ in $K$, then $K$ is a homomorphic image of $\overline{S}=S/\langle z \rangle^{S}$ and the required set of relations hold in $\overline{S}$ for $w=y$. 
\end{proof}

\begin{lemma}\label{l:4 and 3} Let $H<G$ with $H \simeq L_{3}(2)$ and let $t\in \Gamma_{4}(H)$. If $t^2$ inverts $y \in \Gamma_3$, then $\langle t,y \rangle \simeq L_3(2)$.
	
\end{lemma}

\begin{proof} Let $a=t^2$. Choose $K \leq H$ such that $t \in K \sim S_4$. Assume $t \sim (1,2,3,4)$, and take $b \sim (2,4)$, $x\sim (1,4,2)$, $z \sim (2,3,4)$. Then $t=bxz$ and the following relations defining $S_4$ hold: $$\chi_3=\big \{b^2,x^3,z^3,x^bx,z^bz,(xz)^2 \big \}.$$
	
By Lemma \ref{l:f36} we may assume $\langle t,y \rangle \simeq F_{36}$. Note that $b$ inverts $t$. By Lemma \ref{l:f36 and 2} there is $w \in \Gamma_3$ such that relations $\chi_2=\chi_2(b,t,w)$ hold. By Lemma \ref{l:s3 and comm} a subgroup $\langle b,x,z,w \rangle$ is a homomorphic image of $$G(i)=\big \langle b,x,z,w \; \big | \; \{w^3,(twz)^i\} \cup \chi_1(x,w) \cup \chi_1(z,w) \cup \chi_2 \cup \chi_3 \big \rangle.$$ Computations show that $|G(i)| \leq 24$ for any $i\in\{4,5,6,7\}$, and hence $w=1$, a contradiction. 
\end{proof}

\begin{prop}\label{cl32}
Let $a$ be an involution from a subgroup $W$, isomorphic to $L_3(2)$. Then $C_G(a)$ is a $2$-group. 
\end{prop}

\begin{proof}  Assume the contrary and take an element $w$ of order $3$ in $C_G(a)$.

Let $W=\langle t,x \mid \beta \rangle$, where $\beta=\{t^4,x^3,(t^2x)^2,(tx)^7\}$, and define $a=t^2$, $y=x^t$,  $u=xyxy^{-1}xyx$. Note that $a$ inverts the elements of the set $$\Sigma =\{u,u^t\} \subseteq \Gamma_{3},$$ and $(xu)^2=(xu^t)^2=1$. 

We may assume without loss of generality that $$[x,w]=1.$$ Indeed, consider $\langle a,x,w \rangle$, and let $v=[x,w]^2$. By Lemma \ref{l:s3 and comm} $[v,w]=1$ and $v^a=v^{-1}$. So if the order of $v$ is $3$, then by Lemma \ref{l:4 and 3} $\langle t,v \rangle \simeq L_3(2)$, and we may change $x$ to $v$ and $W$ to $\langle t,v \rangle$. If $v=1$, then change $w$ to $xx^wxw^{-1} \in C_G(\langle a,x \rangle)$.

If $z \in \Sigma$, then, by Lemma \ref{l:s3 and comm}, $\langle a,x,z,w \rangle$ is a homomorphic image of $$G(i)=\big \langle a,x,z,w \; \big | \; \gamma  \cup \chi_1(z,w) \cup \{(axwz)^{i}\} \big \rangle$$  were $$ \gamma=\{a^2,x^3,z^3,w^3,(xz)^2,(ax)^2,(az)^2,[a,w],[x,w] \}.$$ Computations show that $|G(i)| \leq 24$ for $i \not = 6$, and $G(6) \simeq (A_4 \times A_4):2$. It is straightforward to check that $[z,w]^2=1$ and $x=z^{-1}wxwz^{-1}w$. 

Let $v=u^{t}$, then $\langle W,w \rangle$ is a homomorphic image of $$ G(i_1,i_2,i_3)=\big \langle t,x,w \mid B \cup \delta_{1} \cup \delta_{2} \cup \epsilon(i_1,i_2,i_3) \big \rangle$$ where $$ \delta_{1}=\{w^3,[a,w],(uw)^6,[u,w]^2,(vw)^6,[v,w]^2, [x,w]\},$$ $$ \delta_{2}=\{x^{-1}u^{-1}wuwu^{-1}w,x^{-1}v^{-1}wvwv^{-1}w\},$$ and $$\epsilon(i_1,i_2,i_3)=\{ (v^uw)^{i_1},(av^uw)^{i_2},(auv^{-1}w)^{i_3} \}.$$ Computations show that $G(7,4,7) \simeq A_8$, which is not possible. Since the order of $G(i_1,i_2,i_3)$ divides $3$ for other parameters $i_1,i_2,i_3 \in \{4,5,6,7\}$, we have reached a contradiction.
\end{proof}

\begin{lemma}\label{l32+} $G$ has a subgroup isomorphic to $2^3:L_3(2)$ or $L_3(4)$.
\end{lemma}

\begin{proof} By Statement \ref{fsg}, $G$ has a subgroup $H$ isomorphic to $A_5$ or $L_2(7)$ such that $\Gamma_2(H) \subseteq \Delta$. In the first case by Lemmas \ref{l:D10}, \ref{l:a7} and Statement \ref{p:s5} $G$ has a subgroup isomorphic to $L_3(4)$. So we further assume that $G \geq H \simeq L_3(2)$ and identify $H$ with $\langle a,x \mid \rho \rangle$, where $$\rho = \{ a^2,x^3,(ax)^7,[a,x]^4 \}.$$

Let $c=x^t$, then $\langle a,c \rangle \simeq S_4$ with $a\simeq (1,2)$ and $c\simeq (2,3,4)$. Denote $v=(ac)^2, s=a^{ca}$ so that $V = \langle v,v^c \rangle \simeq O_2(S_4)$. Then $N_H(V) = \langle a,c \rangle$. If $C_G(V)=V$ then by Lemma \ref{l:s4} $C_G(V_1) > V_1$. Since all involutions of $H$ are conjugated, we may assume that $C_G(V) >V$. By Lemma \ref{l:s4} mod Lemma \ref{l:a7} and Statement \ref{cl32} there is an involution $w \in G \setminus H$ such that one of the following holds:

\vspace{1mm}

{\bf (1)} $\langle a,c \rangle$ centralizes $w$.

Then $\langle a,x,w \rangle$ is a homomorphic image of $$G(i_1,i_2,i_3)=\big \langle a,x,w  \; \big | \; \rho \cup \{w^2,[w,a],[w,c],(xw)^{i_1},[x,w]^{i_2},(axw)^{i_3}\} \big\rangle.$$
Computations show that $G(6,4,7) \simeq G(6,6,7) \simeq V:L_3(2)$, where $|V| =2^6$, and $|G(i_1,i_2,i_3)| \leq 168$ for other parameters. The largest homomorphic image of $G(6,4,7)$ without elements of order $8$ is $2^3:L_3(2)$.

{\bf (2)} There exists a subgroup $W \simeq C_{2} \times C_{2}$ such that $w \in W$, $W \leq C$,  $W \not \leq H$, $H \leq N_G(W)$ and $c$ acts on $W$	fixed point freely. Therefore the following relations hold $$\sigma=\big \{w^2,[w,v],[w,v^c],(cw)^3,[a,w] \big \},$$ so that $\langle a,c,w \mid \sigma \rangle \simeq 2^4:S_3$. Note that $|v^c \cdot a^x | = 3$ in $H$. So by \cite[Lemma 2.2]{m2020e} $\langle v^c,a^x,w \rangle $ has no elements of order $7$. It follows that $\langle a,x,w \rangle$ is a homomorphic image of $$G(i_1,i_2,j)=\big \langle a,x,w \; \big | \;  \rho \cup \sigma \cup \{(axw)^{i_1},(xw)^{i_2},(wv^ca^x)^j \} \big \rangle,$$
where $i_1,i_2,j \in \{4,5,6,7\}$, and $j \not =7$. Computations show that $G(5,7,5) \simeq L_3(4)$, $G(4,7,4) \simeq 2^3:L_3(2)$, and $G(7,6,6) \simeq 2^6:L_3(2)$; while for other parameters we have $|G(i_1,i_2,i_3):\langle a,x \rangle|=1$.
\end{proof}

For any subset $M$ of $G$ we define $$M^+ =\big \{x \in M \; | \; \exists H < G \text{ such that } x \in H \simeq L_3(2) \big \}$$ and $M^-=M\backslash M^+$. 

\begin{lemma}\label{noF42}
 $G$ has no subgroups isomorphic to $F_{42}$. 
\end{lemma}

\begin{proof} Assume $F_{42}=\langle t,x \rangle \leq G$, where $t \in \Gamma_2$ and $x \in \Gamma_3$.
By Statement \ref{cl32} $t \in \Gamma_2^-$. By \cite[Statement 1]{m2020e} there is a unique involution $u$ in $C_G(x)$. Moreover $u$ is conjugated with $t$, and so $u \in \Gamma_2^-$.

By \cite[Lemma 4.3]{m2020e} $x$ cannot sit in a subgroup isomorphic to $A_4$.

Let $a \in \Delta^+$, and consider $\langle a,x \rangle$. By \cite[Lemma 2.1]{m2020e} and Statement \ref{cl32} we have $x^a=x^{-1}$. Indeed, if $|ax|=6$ and order of $[a,x]$ is even there is a subgroup of $\langle a,x \rangle$ isomorphic to $A_4$, which contains $x$; if order of $[a,x]$ is odd, then there is an element of order $3$ in $C_{\langle a,x \rangle}(a)$, which is not possible.

So for every $a \in \Delta^{+}$ we have $x^a=x^{-1}$. Take $W \leq G$ such that $W \simeq L_3(2)$ and choose $a,b \in \Gamma_2(W)$ with $ab \in \Gamma_4$. Then $x^{ab}=x$ and we conclude that $12 \in \omega(G)$, a contradiction.
\end{proof}

\begin{lemma}\label{l34} $G$ has a subgroup isomorphic to $L_3(4)$. 
\end{lemma}

\begin{proof} Assume the contrary. Then by Lemma \ref{l32+} there is a subgroup $H \simeq V:L_3(2)$, where $V\simeq 2^3$. If $v \in V$, then $C_G(v)$ has an element of order $3$. Therefore, by Statement \ref{cl32}, $v\in \Delta^-$.
	
By Proposition \ref{s:arxiv} there is no subgroup isomorphic to $A_4$ that contains $v$. By Lemma \ref{2+3-A4}, for every $x \in \Gamma_3$ we have $(xt)^6=[x,t]^p=1$, where $p$ is odd and, by Lemma \ref{noF42}, $p \not = 7$. Since $C_G(v) >V \simeq 2^3$, then by Lemma \ref{l:D10}, $p \not = 5$. Repeating arguments of Lemma \ref{l:3and3}: if $v$ inverts $x$ of order $3$, then $x \in O_3(G)$ and this case was considered in Paragraph 5. Therefore $v \in \Delta^+$, a contradiction.
\end{proof}
	
\end{section}

\begin{section}{$L_3(4)$-subgroup and Theorem proof}
	
Throughout this section we assume $L_3(4) \simeq H \leq G$. Let $i \in \Gamma_2(H)$, and $C=C_G(i)$.
Then $C_H =C_H(i)$ is a subgroup of $C$ of order $2^6$ with the center $V=Z(C_H) \simeq 2^2$. Let $N=N_H(C_H(i))$, and choose $r \in \Gamma_3 (N)$ so that $N =\langle C_H,r \rangle$. Note that $6 \not \in \omega(H)$, therefore $r$ acts on $C_H$ fixed point freely. In particular, $\langle V,r \rangle \simeq A_4$. Let $j=i^r$ so that $V=\langle i,j \rangle$.
	
\begin{lemma}\label{l:VinO2} $V \subseteq O_2(C)$. 
\end{lemma}
	
\begin{proof}
Assume that $V$ is not contained in $O_{2}(C)$. Then by Baer-Suzuki theorem \cite{bs2014e} there is $t\in \Gamma_2(C_G(i))$ such that $i^rt \in \Gamma_3$. Note that $\langle i,r,t \rangle$ is a homomorphic image of a group $$G(i_1,i_2,i_3)=\big \langle i,r,t \mid i^2,r^3,t^2,(ir)^3,(ti)^2,(ti^r)^3,(trr)^{i_1},(t^rrt)^{i_2},(irt)^{i_3} \big \rangle.$$ Computations show that $G(7,4,7)\simeq A_7$, which is not possible, and $\langle i,r,t \rangle$ is a homomorphic image of $A_4$ for other parameters. \end{proof}
	
\begin{lemma}\label{l:L34O3} $O_3(C)=1$.
\end{lemma}
	
\begin{proof} Assume $1 \not = x \in O_3(C_G(i))$. Then $\langle x,x^j \rangle$ is a 3-subgroup. Therefore $\langle x,j \rangle$ is a homomorphic image of $3^{1+2}:2$. By Lemma \ref{l:VinO2} this implies $[x,j]=1$. So $[x,V]=1$ and $r$ normalizes $V$, therefore $[x^r,V]=1$ and in particular $x^r \in C$. 
		
Let $k$ be an arbitrary involution in $C_H(i)$, and recall that $\langle k,r \rangle \simeq A_4$. Then $\langle k,k^r,x,x^r,x^{r^2} \rangle$ is a $r$-invariant subgroup of $C_G(i)$. It follows that $R=\langle k,r,x \rangle$ is a $\{2,3\}$-group. By \cite[Statement 1]{mamontov2013e} $k,k^r \in O_2(R)$. Hence $k$ inverts no elements of order $3$ from $S = \langle k,x \rangle$. On the other hand $x \in O_3(R)$ and therefore $S$ is a homomorphic image of $3^{1+2}:2$; this is possible only when $[k,x]=1$.
		
Therefore $x$ centralizes all involutions from $C_H(i)$. There are two non commuting involutions in $C_H(i)$, and so $12 \in \omega(G)$, a contradiction.
\end{proof}
	
\begin{lemma}\label{THproof} There is a contradiction, proving Theorem.
\end{lemma}
	
\begin{proof} The subgroup $C$ has exponent $12$, so its $2$-length is $\leq 2$ by \cite{bruh}. By Lemma \ref{l:L34O3} $O_3(C)=1$ and hence $C_C(O_2(C)) \subseteq O_2(C)$ by \cite[Lemma 3]{lmmj2014e} and by Lemma \ref{l:shunk}, $C$ is infinite. Therefore $O_2(C)$ is infinite.
		
It follows that $P=\langle O_2(C), C_H(i) \rangle$ is an infinite $2$-subgroup, and $C_H(i)$ is a finite subgroup of $P$. By Shunkov Theorem \cite[Lemma 2]{l34e} $C_P(C_H(i))$ is infinite, and by \cite{hallkulatilatika,kargapolove} it contains an infinite abelian subgroup. So there is an involution $u \not \in H$ such that $[u,C_H(i)]=1$.

We introduce the following sets of relations:
\begin{itemize}
\item  $\alpha= \{ x^3,y^2,(xy)^7,[x,y]^4 \}$ defining $L_3(2)$;
\item $\beta =\{t^2,t^yt,t^{[x,y]}t \}$, which state that an involution $t$ commutes with a Sylow $2$-subgroup of $L_3(2)$;
\item $\gamma_t(i_1,i_2,i_3,j_1) = \{(xt)^{i_1},(xyt)^{i_2},(xytxt)^{i_3},(ty^{x^2y})^{j_1} \}$. \\These relations were used in \cite{l34e}, where it is shown that they define a finite group $\langle x,y,t \rangle$, when $i_1,i_2,i_3,j_1 \in \{3,4,5,7\}$ and $j_1 \not = 7$. 	
\end{itemize}
		
It is shown in \cite[Lemma 8]{l34e} (and can be verified by computations) that $$R = \big \langle x,y,t \; \big | \; \alpha \cup \beta \cup \gamma_t(7,5,7,5) \,\big\rangle$$ is such that $H=R/Z(R) \simeq L_3(4)$, and so $H$ can be identified with a subgroup $H$ in $G$. Denote $t_1=t^{t^x}$ and $t_2=t^{t^{x^2}}$. Then we can verify that $Z(R)= \langle z \rangle^R$, where $z=(t_1yy^x)^2t$, and 
$C_H(t) = \langle y,y^x,t,t_1,t_2 \rangle$.
		
Let $v=ut$, then $\langle H,u \rangle$ is a homomorphic image of  $$J(i_1,i_2,i_3,j_1,i_4,i_5,i_6,j_2)= \big
\langle x,y,t,u \mid \alpha \cup \beta \cup \gamma_t(7,5,7,5)  \cup \delta \cup \{z\} \big \rangle,$$ where $\delta=\delta(i_1,i_2,i_3,j_1,i_4,i_5,i_6,j_2)$ is the set $$ \{u^2,[u,y],[u,y^x],[u,t],[u,t_1],[u,y_2] \} \cup \gamma_u(i_1,i_2,i_3,j_1) \cup \gamma_v(i_4,i_5,i_6,j_2) $$
and the parameters are in $\{4,5,6,7\}$. Moreover $j_1\not =7$ and $j_2 \not = 7$, since $[x,y]^2 \cdot y^{x^2y} \in \Gamma_3$ and $u$ centralizes $[x,y]^2$ by \cite[Lemma 2.2]{m2020e}. Computations show that $$|J(i_1,i_2,i_3,j_1,i_4,i_5,i_6,j_2) : \langle x,y,t \rangle | \leq 2 $$ for any choice of parameters. Therefore $u \in H$: a contradiction.
\end{proof}
	
\end{section}

\bibliographystyle{unsrt}

\begin{thebibliography}{99}

\bibitem{OCn}
R.Brandl, W.Shi.
\newblock {\it Finite groups whose element orders are consecutive integers.}
\newblock J. Algebra, {\bf 143}(2), (1991), 388--400.

\bibitem{bruh}
E.~G.~Bruhanova.
\newblock {\it On $2$-length and $2$-period of a finite soluble group.}
\newblock Algebra and Logic, {\bf 18}(1), (1979), 5--20.

\bibitem{fin_review}
M.~A.~Grechkoseeva, A.~V.~Vasil’ev. 
\newblock {\it On the structure of finite groups isospectral to finite simple groups.}
\newblock J. Group Theory, {\bf 18}(5), (2015), 741--759.

\bibitem{hallkulatilatika}
P.~Hall, C.~R.~Kulatilaka.
\newblock {\it A property of locally finite groups.}
\newblock J. London Math. Soc., {\bf 39}, (1964), 235--239.

\bibitem{higfree}
G.~Higman.
\newblock {\it Groups and rings having automorphisms without non-trivial fixed elements.}
\newblock J. London Math. Soc., {\bf 32}, (1957), 321--334.

\bibitem{lmmj2014e}
E.~Jabara, D.~V. Lytkina, A.~S. Mamontov, V.~D. Mazurov.
\newblock {\it Groups whose element orders do not exceed 6.}
\newblock Algebra and Logic, {\bf}53(5), (2014), 365--376.

\bibitem{jlm2014e}
E.~Jabara, D.~Lytkina, A.~Mamontov.
\newblock {\it Recognizing ${M}_{10}$ by spectrum in the class of all groups.}
\newblock Int. J. of Algebra and Computation, {\bf 24}(2), (2014), 113--119.


\bibitem{l34e}  E.~Jabara, A.~Mamontov.
\newblock {\it Recognizing $L_3(4)$ by the set of element orders in the class of all groups.}
\newblock  Algebra and Logic, {\bf 54}(4), (2015), 279--282.

\bibitem{kargapolove}
M.~I.~Kargarpolov.
\newblock {\it On the problem of O.~Ju.~Schmidt.}
\newblock  Siberian Math. J., {\bf 4}, (1963),  232--235.
	
\bibitem{lyku1}
D.~V. Lytkina, A.~A. Kuznetsov.
\newblock {\it Recognizability by spectrum of the group ${L}_2(7)$ in the class of all groups.}
\newblock Sib. Electronic Math. Reports, {\bf 4}, (2007), 300--303.

\bibitem{rev_per}
D.~V.~Lytkina, V.~D.~Mazurov.
\newblock {\it Groups with given element orders.}
\newblock J. Sib. Fed. Univ. Math. Phys., {\bf 7}(2), (2014), 191--203.

\bibitem{bs2014e}
A.~S. Mamontov.
\newblock {\it The Baer–Suzuki theorem for groups of $2$-exponent $4$.}
\newblock Algebra and Logic, {\bf 53}(5), (2014), 422--424.

\bibitem{m2020e}
A.~Mamontov
\newblock {\it On periodic groups isospectral to $A_{7}$.}
\newblock  Siberian Math. J., {\bf 61}(1), (2020), 109--117.

\bibitem{maz1-5}
V.~D.~Mazurov.
\newblock {\it On the groups of period 60 with prescribed orders of elements.}
\newblock Algebra and Logic, {\bf 39}(3), (2000), 189--198.

\bibitem{maze}
V.~D. Mazurov.
\newblock {\it Infinite groups with abelian centralizers of involutions.}
\newblock Algebra and Logic, {\bf 39}(1), (2000), 42--49.

\bibitem{pernotfin}
V. D. Mazurov, A. Yu. Ol'shanskii, A. I. Sozutov.
\newblock {\it Infinite groups of finite period.}
\newblock Algebra and Logic, {\bf 54}(2), (2015),  161--166.

\bibitem{neu1937}
B.~H.~Neumann.
\newblock {\it Groups whose elements have bounded orders.}
\newblock J. London Math. Soc., {\bf 12}, (1937), 195--198.

\bibitem{sane}
I.~N.~Sanov.
\newblock {\it Solution of Burnside's problem for $n=4$} (in {R}ussian).
\newblock Leningrad State University Annals (Uchenyi Zapiski) Math. Ser., {\bf 10}, (1940), 166--170.
	
\bibitem{shu1972e}
V.~P.~Shunkov.
\newblock {\it Periodic groups with an almost regular involution.}
\newblock Algebra and Logic, {\bf 11}(4), (1972), 260--272.		

\bibitem{zhma1999e}
A.~Kh. Zhurtov, V.~D. Mazurov.
\newblock {\it On recognition of the finite simple groups ${L}_2(2^m)$ in the class of all groups.}
\newblock Siberian Math. J., {\bf 40}(1), (1999), 62--64.

\bibitem{zh2000e}
A.~Kh. Zhurtov.
\newblock {\it Regular automorphisms of order 3 and Frobenius pairs.}
\newblock Siberian Math. J., {\bf 41}(2), (2000), 268--275.

\bibitem{gap} The {\sc Gap}: groups, algorithms, and programming, vers. 4.10.2 (2019), {\tt http:// www.gap-system.org}.

\end{thebibliography}

\vspace{8mm}

{\sc Enrico Jabara}

{\sc DFBBC Universit\`{a} di Venezia}

{ Dorsoduro 3484/D - 30123  Venezia 30123 - ITALY}

{\tt email: jabara@unive.it}

\vspace{3mm}

{\sc Andrey Mamontov}

{\sc Novosibirsk State University}

{\sc Pirogova Str. 2 - Novosibirsk 630090 - RUSSIA }

{\tt email: andreysmamontov@gmail.com}
\end{document}